\documentclass[a4paper,11pt,colorlinks=true,linkcolor=RedViolet,urlcolor=RoyalPurple,citecolor=Red,bookmarksopen=true]{article}
\usepackage{mathrsfs}
\usepackage{graphicx}

\usepackage{amsmath,amsfonts,amssymb,amsthm,epsfig,epstopdf, dsfont, faktor, tensor, slashed, mathtools, cancel, tabularx, tcolorbox, color, soul,  wasysym, authblk, titling, enumerate, epigraph ,url, array, stmaryrd, bold-extra}
\usepackage{quiver}\usepackage{quiver}

\renewenvironment{thebibliography}[1]{
  \begin{oldthebibliography}{#1}
    \setlength{\itemsep}{0em}
    \setlength{\parskip}{0.2em}
}
{
  \end{oldthebibliography}
}

\usepackage{bookmark}
\usepackage[toc,page]{appendix}

\def\Q{{\mathbb Q}}
\def\B{{\mathbb B}}
\def\P{{\mathbb P}}
\def\R{{\mathbb R}}
\def\Z{{\mathbb Z}}
\def\ZK{{\Z_{(K)}}}
\def\PzZK{{\P(z;\ZK)}}

\def\PzwZK{{\P(z,w;\ZK)}}
\def\Zh{{\Z[[h]]}}
\def\Qh{{\Q[[h]]}}
\def\Rh{{\R[[h]]}}
\def\Rxh{{\R[x][[h]]}}
\def\Rzh{{\R[z][[h]]}}
\def\Rztwoh{{\R[z^2][[h]]}}
\def\Zpos{{Z^{\geq0}}}
\def\ZKh{{\ZK[[h]]}}
\def\ZKph{{\Z_{(K')}[[h]]}}
\def\PzZKh{{\P(z;\ZK)[[h]]}}
\def\PalphaZKh{{\P(\alpha;\ZK)[[h]]}}

\def\PzwZKh{{\P(z,w;\ZK)[[h]]}}
\def\G{{\mathbb G}}
\def\N{{\mathbb N}}

\def\sA{{\mathbf{A}}}
\def\sB{{\mathbf{B}}}
\def\sC{{\mathbf{C}}}
\def\sD{{\mathbf{D}}}
\def\cN{{\mathcal N}}
\def\cL{{\mathcal L}}
\def\cLb{{\overline{\cL}}}

\def\cP{{\mathcal P}}
\def\cT{{\mathcal T}}
\def\eK{{\equiv_K}}
\def\bfz{{\bf z}}
\def\neK{{\not\equiv_K}}
\def\ePK{{\equiv_{\Phi_K}}}
\def\ePKp{{\equiv_{\Phi_{K'}}}}

\def\p{{\mathfrak{p}}}
\def\f{{\mathfrak{f}}}
\def\c{{\mathfrak{c}}}

\def\bfz{{\bf z}}
\def\bfr{{\bf r}}
\def\bfn{{\bf n}}
\def\bfA{{\bf A}}
\def\bfa{{\bf a}}
\def\bfm{{\bf m}}
\def\bfv{{\bf v}}
\def\bfx{{\bf x}}
\def\bfb{{\bf b}}
\def\bfD{{\bf D}}

\def\vol{{\text{vol}}}

\def\sgn{{\text{sgn}}}

\def\trap{{\trapeziumfilled}}
\def\dtrap{{\trapeziumfilledneg}}
\def\fbar{{\overline{f}}}
\def\cbar{{\overline{c}}}

\newcommand{\Bint}{\mathrlap{\:B}\int}
\newcommand{\Bintt}{{\mathrlap{\kern1pt{}B}\int}}
\newcommand{\Biint}{\mathop{\mathchoice
  {\Biintaux\displaystyle}
  {\Biintaux\textstyle}
  {\Biintaux\scriptstyle}
  {\Biintaux\scriptscriptstyle}
}\displaylimits}

\newcommand{\Biintaux}[1]{%
  \ooalign{%
    \hfil$#1 B$\hfil\cr
    $#1 \iint$\cr
  }%
}
\newcommand{\rectangle}{{%
  \ooalign{$\sqsubset\mkern3mu$\cr$\mkern3mu\sqsupset$\cr}%
}}

\newcommand{\squarefilled}{%
  \begin{tikzpicture}[x=1ex,y=1ex,baseline=-0.1ex]
    \fill (0,0) rectangle (1,1);
  \end{tikzpicture}%
}

\newcommand{\trapeziumfilled}{%
  \begin{tikzpicture}[x=1.2ex,y=1.2ex,baseline=-0.1ex]
    \fill (0,0) -- (1,0) -- (1,1) -- (0,0.7) -- cycle;
  \end{tikzpicture}%
}

\newcommand{\trapeziumfilledneg}{%
  \begin{tikzpicture}[x=1.2ex,y=1.2ex,baseline=-0.1ex]
    \fill (0,1) -- (1,1) -- (1,0) -- (0,0.3) -- cycle;
  \end{tikzpicture}%
}

\usepackage[dvipsnames]{xcolor}

\theoremstyle{plain}
\newtheorem{thm}{Theorem}[section]
\newtheorem{prop}[thm]{Proposition}
\newtheorem{lem}[thm]{Lemma}
\newtheorem{exmp}[thm]{Example}
\newtheorem{defn}[thm]{Definition}
\newtheorem{conj}{Conjecture}

\newtheorem*{cor}{Corollary}
\newtheorem*{rem}{Remark}

\numberwithin{equation}{section}

\usepackage[english]{babel}
\usepackage[a4paper,top=1.5cm,bottom=2.5cm,left=1.5cm,right=1.5cm,marginparwidth=1.75cm]{geometry}

    \title{\textbf{Quantum Pick's Theorem in two dimensions}\thanks{This research was supported in part by grant 2020242 from the United States-Israel Binational Science Foundation (BSF), Jerusalem, Israel.}}
\author{{Moris Metin Bardavid \& Ruth Lawrence}\thanks{Hebrew University of Jerusalem; \tt{morisme.bardavid\textbf{[at]}mail.huji.ac.il}, \tt{ruthel.naimark\textbf{[at]}math.huji.ac.il}}}
\begin{document}
\maketitle
\begin{abstract}
   Consider an $n$-dimensional sum of the form $S_K=\sum q^{f(\bfz)}$ over lattice points $\bfz$ which lie in $K\cP$, the $K$ scaling of an $n$-dimensional rational polytope $\cP$. We formulate a conjecture on the structure of $S_K$ considered for $K$ prime as a power series in $h=q-1$ modulo the $K$-th cyclotomic polynomial, and in particular the existence of a decomposition in powers of the Gauss sum $\G$, up to the $n$-th, whose coefficients are almost independent of $K$, where the leading coefficient is related to the volume of $\cP$ and the coefficient of $\G^r$ has contributions from the codimension $r$ skeleton of $\cP$.  The conjecture is proved for $n=1,2$.
\end{abstract}

\section{\textbf{\textsc{Introduction and main result}}}

\begin{conj}\label{quantumpickconjecture}  Suppose that $f(\bfz)$ is a rational quadratic function of $\bfz\in\R^n$ with positive definite homogeneous part and $\cP$ is an $n$-dimensional rational polytope. Then the integer lattice sum,
\begin{align*}
    S_K(\cP) :=\sum_{\bfz \in K \cP\cap\Z^n } q^{f(\bfz)}
\end{align*}
satisfies the following properties,
\begin{enumerate}
    \item As a formal power series in $h=q-1$, $S_K(\cP)\in\ZKh$ for all but finitely many primes $K$.
    \item There exists a unique level $n$, $\G$-decomposition of $S_K$, that is there exist $\cN\in\N$ and power series $Z_K^{(r)}(\cP)\in\Q[[h]]$ (for $r=0,1,\ldots,n$ and all prime $K$ coprime to $\cN$) whose coefficients depend on $K$ only modulo $\cN$ and such that for all sufficiently large prime $K$ coprime to $\cN$,
    \[
         S_K(\cP)  \equiv_{\Phi_K} \sum_{r=0}^n Z_K^{(r)}(\cP) \G^r
    \]
    where $\G=\sum\limits_{r=0}^{K-1}q^{r^2}$ is the Gauss sum.
    \item $Z_{K}^{(n)} = \left( \frac{(\Delta)_K}{K} \right) q^{f_0}  \cdot \vol(\cP)$ where $\Delta\in\Q$ is the discriminant of the homogeneous part of $f$, $(\Delta)_K$ is the residue of $\Delta\in\ZK$ mod $K$, $\vol(\cP)$ denotes the volume of $\cP$ and $f_0$ denotes the value of $f$ at its critical point.
    \item $Z_K^{(r)}(\cP)$ can be decomposed as a sum of contributions corresponding to codimension $r$ faces of $\cP$, each term depending only on that face and local data on how it sits in $\cP$.
\item For each $i$ coprime to $\cN$, the power series $Z^{(r)}_K(\cP)$ for sufficiently large prime $K\equiv_{\cN}i$ is the asymptotic expansion around $q=1$ of a holomorphic function $Z_{\infty,i}^{(r)}(q)$ of $q$ defined either for $|q|>1$ or for $|q|<1$, which has almost modular properties (in the sense of the generalised modularity conjecture \cite{GMC} ).
\end{enumerate}
\end{conj}

\noindent In this paper we prove parts (1)-(4) of the conjecture for $n=1,2$. Sections 2,3  are preliminaries introducing the notation. Section 4 develops results on Bernoulli numbers and symbolic integration which will be used later. In section 5, we deal with the one-dimensional case and in section 6 with the two-dimensional case.

\medskip
\noindent The motivation for the paper comes from the generalized modularity conjecture \cite{GMC} on quantum 3-manifold invariants. The value of a quantum invariant of a closed compact oriented 3-manifold at a $K$-th root of unity can be written as a ratio of $q$-sums over integer points in a $K$-scaled rational polytope, whose summand is a product of $q$ powers, theta-nets and quantum $6j$ symbols. For particular manifolds, for example Seifert fibred manifolds with three fibres, the summand may be a $q$-power by a quadratic expression alone and then the known almost modularity properties of quantum invariants in these special cases (see for example \cite{LZ}) would relate to the properties in Conjecture 1.  This paper provides the first step in the investigation of the structure of general such $q$-sums, by dealing with the $\G$ decomposition in one and two dimensional cases.

\medskip\noindent
Note that the structure of a sum of $q$-exponentials of a {\sl linear} integral form $\lambda:\Z^n \rightarrow \Z$ (as opposed to a quadratic form as in the current paper) over integer points in a scaled polytope, was studied in \cite{CH} for lattice polytopes, while an extension into rational polytopes (and further insight) is provided in \cite{BK}.

\section{\textbf{\textsc{Properties of}} $\ZK$}

Let $K$ be an odd prime.

\medskip
\noindent{\bf $K$-adic order:}  Any non-zero $x\in\Q$ can be uniquely written in the form $x=K^m\cdot\frac{a}{b}$ where $a\in\Z^\times$ and $b\in\N$ are coprime integers not divisible by $K$ and $m\in\Z$. The exponent $m$ will be denoted by $v_K(x)$ and defines the $K$-adic order map $v_K:\Q^\times\longrightarrow\Z$ by $x\mapsto{}m$, which is a group homomorphism from $(\Q^\times,\times)$ to $(\Z,+)$.

\medskip\noindent{\bf $K$-adic integers:} Denote by $\Z_{(K)}$ the subset of $\Q$ on which $v_K$ is non-negative, that is, rationals (including zero) which can be written in the form $\frac{a}{b}$ where $b\in\N$ is not divisible by $K$ and $a\in\Z$. Then $\Z_{(K)}$ forms a ring. For $i\in\N$, we will say that $x\in\Z_{(K)}$ is divisible by $K^i$ when $K^{-i}x\in\Z_{(K)}$; equivalently the condition is that $v_K(x)\geq{}i$. For $x,y\in\Q$, we write $x\equiv_{K^i}y$ to mean that $x,y\in\ZK$ and $x-y$ is divisible by $K^i$ in $\ZK$.

\medskip\noindent{\bf Remainder map:} For any $i\in\N$ and $x\in\Z_{(K)}$, let $(x)_{K^i}$ denote the unique integer in $\{0,1,\ldots,K^i-1\}$ such that $x-(x)_{K^i}$ is divisible by $K^i$ in $Z_{(K)}$. When $a$ is an integer, $(a)_{K^i}$ is the remainder of $a$ on division by $K^i$. More generally, for $b\in\N$ coprime to $K$ and $a\in\Z$,
\[\left(\frac{a}{b}\right)_{K^i}=(a\cdot\overline{b})_{K^i}\]
where $[\overline{b}]$ is the inverse to $[b]$ in $\Z/K^i\Z$.

\begin{lem}\label{combinationbelongs} If $\alpha\in\ZK$ and $n\in\N$ then
$\binom{\alpha}{n}\equiv\frac{1}{n!}\alpha(\alpha-1)\dotsm(\alpha-n+1)\in\ZK$.
\end{lem}
\begin{proof} When $\alpha\in\Z$, immediately $\binom{\alpha}{n}\in\Z\subset\Z_{(K)}$, so we assume that $\alpha\not\in\Z$. In order to prove that $\binom{\alpha}{n}\in\Z_{(K)}$, it suffices to show that $v_K\left(\binom{\alpha}{n}\right)\geq0$. Write $\alpha = \frac{a}{b}$ where $a\in\Z$ and $b\in\N$ are coprime and $(b,K)=1$. Then
\[
\binom{\alpha}{n}=\frac{\alpha(\alpha -1)\dotsm(\alpha-n+1)}{n!}
    = \frac{1}{b^n}\cdot\frac{(a)(a-b)\dotsm(a-(n-1)b)}{n!}
    \]
This is non-zero because $\alpha\not\in\N$. Since $(b,K)=1$, thus the only $K$-factors come from the last factor and
\begin{align*}
 v_K\left(\binom\alpha{n}\right)&=v_K(a)+v_K(a-b)+\cdots+v_K(a-(n-1)b)-v_K(n!)\\
     &=\sum_{i=1}^\infty \left(\#\left\{t\Big|\,0\leq{}t\leq(n-1),\quad{}K^i\,\big|\,(a-tb)\right\}
                        -\#\left\{t\Big|\,1\leq{}t\leq{}n,\quad{}K^i\,\big|\,t\right\}\right)\\
    &=\sum_{i=1}^\infty \left(\#\left\{t\Big|\,0\leq{}t\leq(n-1),\quad{}t\equiv(\alpha)_{K^i}\hbox{ mod $K^i$}\right\}
                        -\#\left\{t\Big|\,1\leq{}t\leq{}n,\quad{}t\equiv0\hbox{ mod $K^i$}\right\}\right)\\
    &=\sum_{i=1}^\infty \left(\left\lfloor \frac{n-1}{K^i} \right\rfloor+\delta_{(\alpha)_{K^i}\leq(n-1)_{K^i}}
                        -\left\lfloor \frac{n}{K^i} \right\rfloor\right)
\end{align*}
where in the last line $\delta_T$ is a Boolean delta symbol which is 1 if $T$ is true and 0 otherwise. To finish the proof, we need to show that this sum is non-negative. Indeed we claim that each term in the sum (that is for each $i$) is independently non-negative. Observe that
\[\left\lfloor\frac{n-1}{K^i}\right\rfloor-\left\lfloor\frac{n}{K^i}\right\rfloor
= \begin{cases}
        -1 &\hbox{if $K^i\,\big|\,n$} \\[1.5ex]
        0  &\hbox{otherwise}
    \end{cases}
\]
while when $K^i\,\big|\,n$, precisely $(n-1)_{K^i}=K^i-1$ so that automatically $(\alpha)_{K^i}\leq(n-1)_{K^i}$ and the Boolean delta symbol will be $+1$ exactly canceling the $-1$ arising from the difference of integer parts.
\end{proof}

\begin{defn}\label{P(z,R)definition} Let $R \subseteq \Q$ be any ring. Define the set
\begin{align*}
    \P(z;R) := \left\{ \sum_{r=0}^\infty a_r \binom{z}{r}\Bigg|\,a_r\in R\hbox{ with $a_r= 0$ for all but finitely many $r$} \right\}
\end{align*}
\end{defn}

\begin{lem}\label{PZKonZK}
    If $p(z)\in\PzZK$ then $p(x)\in\ZK$ for all $x\in\ZK$.
\end{lem}
\begin{proof}
    The proof follows immediately from Definition~\ref{P(z,R)definition} and Lemma~\ref{combinationbelongs}.
\end{proof}

\begin{lem}\label{P(z,R)characterization} Suppose that $p(z)\in\Q[z]$ and $R\subset\Q$ is any ring. Then the following are equivalent
\begin{enumerate}
    \item $p(z)\in\P(z;R)$
    \item $p(n)\in{}R$ for all $n\in\Z$
    \item $p(n)\in{}R$ for all $n\in\N$
\end{enumerate}
 In particular,  $R[z] \subset \mathbb{P} (z; R)$.
\end{lem}
\begin{proof}
 1$\Rightarrow$2: If $p(z)\in\mathbb{P}(z;R)$ then by definition,  $p(z)=\sum_{r=0}^\infty{}a_r\binom{z}{r}$ for some sequence $(a_r)_{r\geq0}$ of elements of $R$.  Since $\binom{n}{r}\in\Z$ for all $n\in\Z$, $r\in\N$, thus $p(n)=\sum_{r=0}^\infty{}a_r\binom{n}{r}\in{}R$ for all $n\in\Z$.

\noindent 2$\Rightarrow$3: Immediate.

\noindent 3$\Rightarrow$1: We prove by induction on the degree $d$ of the polynomial $p(z)$ that if $p(n)\in{}R$ for all
$n\in\N$ then $p(z)\in\mathbb{P}(z;R)$. For $d=0$, that is constant polynomials, this is immediate since
$\binom{z}{0}\equiv1$. Suppose that $p(z)\in{}\Q[z]$ of degree at most $d$($>0$) is such that $p(n)\in{}R$ for all
$n\in\N$. Define the finite difference operator $\Delta$ on $\Q[z]$ by $\Delta{}p(z)=p(z+1)-p(z)$; it decreases the
degree of a polynomial by one and multiplies its leading coefficient by its degree so that on a polynomial of degree
$d$, $\Delta^dp(z)=d!\cdot{}p_d$ where $p_d$ is the coefficient of $z^d$ in $p(z)$. Since $p(n)\in{}R$ for all
$n\in\N$ thus $\Delta^ip(n)\in{}R$ for all $i,n\in\N$. In particular, $i=d$ gives that $d!p_d\in\Z$. Thus
$p(z)-d!p_d\binom{z}{d}$ is a polynomial of degree less than $d$ (since it has been constructed so that the
coefficient of $z^d$ vanishes) whose values on all natural numbers lie in $R$ and so by the inductive hypothesis can
be written as an $R$-linear combination of functions $\binom{z}{r}$, $0\leq{}r<d$.
\end{proof}

\noindent Combining Lemmas \ref{PZKonZK}, \ref{P(z,R)characterization}, we get the following lemma.
\begin{lem}\label{pztoZK}
    If $p(z)\in\PzZK$ and $p(z_0)\in\ZK$ for all $z_0\in\Z$ then $p(z_0)\in\ZK$ for all $z_0\in\ZK$.
\end{lem}

\noindent There are similar results for polynomials in two variables.

\begin{defn} Let $R \subseteq \Q$ be any ring. Define the set
\[\P(x,y;R):= \Bigg\{ \sum_{m,n= 0}^\infty a_{m,n} \binom{x}{m } \binom{y}{n } \Bigg|\, a_{m,n}\in R\hbox{ with $a_{m,n}=0$ for all but finitely many pairs $(m,n)$}\Bigg\}
\]
\end{defn}

\begin{prop}\label{2dPxyRcharacterisation} Suppose $p(x,y) \in \Q [x,y]$ and $R\subset\Q$ be any ring. Then the following are equivalent
\begin{enumerate}
    \item $p(x,y)\in\P(x,y;R)$
    \item $p(x,y)\in{}R$ for all $m,n\in\Z$
    \item $p(x,y)\in{}R$ for all $m,n\in\N$
\end{enumerate}
\end{prop}
\begin{proof}
    The proof is similar to that of Lemma \ref{P(z,R)characterization}, with the change in 3$\Rightarrow$1 that there are now two difference operators $\Delta_xp(x,y)=p(x+1,y)-p(x,y)$ and $\Delta_yp(x,y)=p(x,y+1)-p(x,y)$, which respectively decrease the $x$-degree and the $y$-degree by one. Induction is now by the total degree $d$ of $p(x,y)$. Supposing that $p(x,y)\in{}R$ for all $x,y\in\N$ then $\Delta_x^i\Delta_y^{d-i}p(x,y)=i!(d-i)!p_{i,d-i}\in{}R$ where $p_{i,j}$ is the coefficient of the monomial $x^iy^j$ in $p(x,y)$ and so
    \[p(x,y)-\sum_{i=0}^d\left(i!(d-i)!p_{i,d-i}\binom{x}{i}\binom{y}{d-i}\right)\] will be a polynomial of total degree less than $d$ whose values on $\N\times\N$ lie in $R$, completing the inductive step.
\end{proof}

\section{\textbf{\textsc{Formal Power Series}}}
Throughout this paper we will be working in formal power series in a formal variable $h$, with either rational coefficients (usually in $\Z_{(K)}$ for some prime $K$) or coefficients which are polynomial in an additional variable or variables. Set $q=1+h$.

\medskip\noindent{\bf Powers of $q$:} For any $\alpha\in\R$, the symbol $q^\alpha$ will refer to the formal power series obtained by expansion of the binomial series $(1+h)^\alpha=\sum_{n=0}^\infty\binom{\alpha}{n}h^n\in\R[[h]]$. In particular, thereby $\Q[q,q^{-1}]$ can be considered as embedded in $\Q[[h]]$ as well as $\Z[q,q^{-1}]$ considered as embedded in $\Z[[h]]$.

\begin{lem}\label{qpoweralpha} If $\alpha\in\ZK$ then $q^\alpha\in\ZKh$.
\end{lem}
\begin{proof}
    This is immediate from the definition and Lemma \ref{combinationbelongs} since $\alpha\in\ZK\implies\binom{\alpha}{n}\in\ZK\>,\forall n\in\Z^{\geq0}$.
\end{proof}

\begin{lem} \label{qpowerfz} If $f(z)\in\PzZK$ then $q^{f(z)}\in\PzZKh$.\end{lem}
\begin{proof}
  Suppose $n\in\Z^{\geq0}$.  For any $z_0\in\Z$ then $f(z_0)\in\ZK$ and so by Lemma \ref{combinationbelongs}, $\binom{f(z_0)}{n}\in\ZK$. Since $f(z)$ is a polynomial, thus $\binom{f(z)}{n}$ is a polynomial of degree $n\cdot\deg(f)$ which takes values in $\ZK$ at all integer points, so that by Lemma \ref{P(z,R)characterization}, $\binom{f(z)}{n}\in\PzZK$. Hence $q^{f(z)}=\sum_{n=0}^\infty{}\binom{f(z)}{n}h^n\in\PzZKh$.
\end{proof}

\begin{lem}\label{pzqtoZKh}
    If $\p(z;q)\in\PzZKh$ and $\p(z_0;q)\in\ZKh$ for all $z_0\in\Z$, then $\p(z_0;q)\in\ZKh$ for all $z_0\in\ZK$.
\end{lem}

\noindent There are similar results for two variables.
\begin{lem} \label{qpowerfzw} If $f(z,w)\in\PzwZK$ then $q^{f(z,w)}\in\PzwZKh$.\end{lem}

\begin{lem}\label{pzwqtoZKh}
    If $\p(z,w;q)\in\PzwZKh$ then $\p(z_0,w_0;q)\in\ZKh$ for all $z_0,w_0\in\ZK$.
\end{lem}

\subsection{\textsc{The Cyclotomic Polynomial}}
Denote the $K^{\text{th}}$ cyclotomic polynomial by $\Phi_K(q)=\frac{q^K-1}{q-1}$. In terms of $h$, $\Phi_K(q)=\sum_{r=0}^{K-1}\binom{K}{r+1}h^r$ which is an integer coefficient monic polynomial in $h$ of degree $K-1$, all of whose other coefficients (that is, of $h^n$ for $n<K-1$) are divisible by $K$ and whose constant term is $K$.

\medskip\noindent{\bf Divisibility by $\Phi_K$:} For any ring $R$, a formal power series $\sum_{r=0}^\infty{}a_rh^r\in{}R[[h]]$ is invertible in $R[[h]]$ if, and only if, $a_0\in{}R$ is invertible. In the case of $R=\Z_{(K)}$, the invertibility condition is that the constant term $a_0$ is not divisible by $K$. In particular, $\Phi_k$ is not invertible in $\Z_{(K)}[[h]]$ although it is invertible in $\Q[[h]]$. Thus divisibility of an element of $\Z_{(K)}[[h]]$ by $\Phi_K(q)$ is a non-trivial condition which is equivalent to an infinite number of congruences.

\begin{defn} The formal power series $\sum_{n=0}^\infty{}c_n h^n\in\Z_{(K)}[[h]]$ is said to be divisible by $\Phi_K(q)$ if, and only if, there exists $\sum_{n=0}^\infty a_n h^n\in\Z_{(K)}[[h]]$ such that $\sum_n c_n h^n= \left( \sum_n a_n h^n  \right) \Phi_K (q)$. Equivalently, $\sum_{k=0}^nb_{n-k}c_k\in\Z_{(K)}  $ for all $n\in\N$, where $\{b_n\}_{n=0}^{\infty}$ are the coefficients of the formal power series $\frac{1}{\Phi_K(q)}\in\Q[[h]]$. Furthermore, we will write $p(h)\equiv_{\Phi_K}q(h)$ to mean that $p(h),q(h)\in\Z_{(K)}[[h]]$ and that their difference is divisible by $\Phi_K(q)$ in $\Z_{(K)}[[h]]$.
\end{defn}

\begin{lem}\label{gammaKlemma} If $\gamma \in \Z_{(K)}$ then $q^{\gamma K }\equiv_{\Phi_K}1$. Equivalently, if $a\eK0$ then $q^a\ePK1$.
\end{lem}
\begin{proof} Write $\gamma = \frac{r}{s}$ with $r\in\Z$ and $s\in\N$ coprime and $(s,K)=1$. For all $a\in\Q$, $\frac{q^a-1}{q-1}\in\Q[[h]]$ with constant term $a$. If in addition $b\in\Q^\times$, then $\frac{q^b-1}{q-1}$ is invertible in $\Q[[h]]$ and so $\frac{q^a-1}{q^b-1}$ is a well-defined expression in $\Q[[h]]$, since it can be written as
\[\frac{q^a-1}{q^b-1}=\frac{q^a-1}{q-1}\cdot\left(\frac{q^b-1}{q-1}\right)^{-1}
\]
As elements of $\Q[[h]]$,
\[
    \frac{q^{\gamma{}K} -1}{ q^K-1 } =
 \frac{ q^{rK/s} -1 }{q^{rK}-1} \cdot\frac{ q^{rK}-1 }{q^{K} -1} = \frac{ q^{(r-1) K} + q^{(r-2 )K}+ ... + q^K + 1   }{
 q^{\frac{s-1}{s} rK } + q^{\frac{s-2}{s} rK  } + ... + q^{\frac{1}{s}rK } + 1    }
\]
Since $(s,K)=1$, $\gamma{}K$ and the exponents of $q$ appearing in the numerator and denominator are elements of $Z_{(K)}$ and so $q^{\gamma{}K}$ as well as the numerator and denominator themselves are elements of $\Z_{(K)}[[h]]$. Furthermore, the constant term in the denominator (as a formal power series) is $s\not\equiv0$ (mod $K$) and so the denominator is invertible in $\Z_{(K)}[[h]]$. This proves that $\frac{q^{\gamma{}K} -1}{ q^K-1 }\in\Z_{(K)}[[h]]$. Hence
\[\frac{q^{\gamma{}K}-1}{\Phi_K(q)}=(q-1)\cdot\frac{q^{\gamma{}K}-1}{q^K-1}\in\Z_{(K)}[[h]]
\]
\end{proof}

\begin{prop}\label{modKmodphiK} Suppose $\alpha, \beta \in \Z_{(K)} $. Then $q^{\alpha }\equiv_{\Phi_K}q^{\beta}$ if and only if $\alpha\equiv_K\beta$. In particular, for all $a\in\ZK$ then $q^a\ePK{}q^{(a)_K} $.
\end{prop}

\begin{proof} $(\Leftarrow)$ Suppose that $\alpha \equiv_K\beta$. Put $\gamma=\frac{\alpha-\beta}{K}\in\Z_{(K)}$. By lemma \ref{gammaKlemma}, $q^{\gamma{}K}-1$ is divisible by $\Phi_K(q)$ and thus so is $q^\alpha - q^\beta = q^{\beta} ( q^{\gamma{}K} -1 )$.\newline
$(\Rightarrow)$ Suppose that $q^\alpha\equiv_{\Phi_K}q^\beta $. As a formal power series in $h$, $q^\alpha-q^\beta$ has zero constant term and the coefficient of $h^1$ is $\alpha-\beta$. Since it is divisible by $\Phi_K(q)$, whose constant term is $K$, thus $\alpha-\beta$ is divisible by $K$ in $\Z_{(K)}$.
\end{proof}

\begin{lem} \label{pzqmodPhicharacterisation}
Suppose that $\p(z;q)\in\PzZKh$. Then the following statements are equivalent
\begin{enumerate}
    \item $\p(z;q)\ePK0$ in $\PzZKh$
    \item $\p(z_0;q)\ePK0$ in $\ZKh$ for all $z_0\in\Z$
    \item $\p(z_0;q)\ePK0$ in $\ZKh$ for all $z_0\in\ZK$
\end{enumerate}
\end{lem}
\begin{proof}
    Consider $p(z;q)$ as an element of $\Q[z][[h]]$ and take the quotient $\c(z;q)=\Phi_K^{-1}\cdot\p(z;q)\in\Q[z][[h]]$ which can be expanded as $\c(z;q)=\sum_{n=0}^\infty{}c_n(z)h^n$.

    \noindent(1)$\Leftrightarrow$(2): Then $\p(z;q)\ePK0$  in $\PzZKh$ iff $\c(z;q)\in\PzZKh$ iff $c_n(z)\in\PzZK\>,\forall n\in\Z^{\geq0}$ iff (by Lemma \ref{P(z,R)characterization}) $c_n(z_0)\in\ZK\>,\forall n\in\Z^{\geq0}, z_0\in\Z$ iff $\c(z_0,q)\in\ZKh\>,\forall n\in\Z^{\geq0}$ iff $\p(z_0;q)\ePK0 ,\forall z_0\in\Z$ in $\ZKh$. (1)$\Rightarrow$(3) follows from Lemma \ref{PZKonZK} while (3)$\Rightarrow$(2) is immediate.
\end{proof}

\subsection{\textsc{Gauss Sums}}

Let $\G:=\sum\limits_{z=0}^{K-1} q^{z^2}\in\Z[h]$
and $\tilde\G:=\sum\limits_{z=0}^{K-1}q^{-z^2}\in\Z[[h]]$ denote the two standard Gauss sums.

\noindent Note that $\G$ and $\tilde\G$ both have constant term $K$ as elements of $\Z[[h]]$ and so are not invertible in $\ZKh$.

\begin{lem}\label{qazsum}
    Suppose $a\in\ZK$. Then $\sum\limits_{z=0}^{K-1}q^{za}\ePK\begin{cases}
       K  &\hbox{if $a\equiv_K0$}\\[1.5ex]
       0  &\hbox{otherwise}
    \end{cases}$
\end{lem}
\begin{proof}
    If $a\eK0$ then $za\eK0$ and so by lemma \ref{gammaKlemma}, $q^{za}\ePK1$ for all $z=0,\ldots,K-1$ so that $\sum\limits_{z=0}^{K-1}q^{za}\ePK{}K$.

\noindent If $a\neK0$ then $(a)_K\not=0$ is invertible in $\Z/K\Z$. Hence $(za)_K$ ranges over all values $0,\dots,K-1$ as $z$ ranges over $0,\ldots,K-1$ and so by Lemma \ref{modKmodphiK},
    \[
\sum\limits_{z=0}^{K-1}q^{za}\ePK\sum\limits_{z=0}^{K-1}q^{(za)_K}=\sum\limits_{r=0}^{K-1}q^{r}=h\Phi_K\ePK0
    \]
\end{proof}

\begin{lem}\label{generalizedquadraticgausssum}
\begin{enumerate}
    \item $\G\tilde\G\equiv_{\Phi_K}K\>.$
    \item Suppose $a\in\ZK$. Then
    \[\G_a:=\sum\limits_{z=0}^{K-1} q^{az^2} \equiv_{\Phi_K}\begin{cases}
        \G &\hbox{if $(a)_K\not=0$ is a perfect square in $\Z/K\Z$} \\[1.5ex]
        K  &\hbox{if $a\equiv_K0$}\\[1.5ex]
        -\G&\hbox{otherwise}
    \end{cases}\]
\end{enumerate}
\noindent In particular, $\G_{-1}=\tilde\G$ while when $a\not\equiv_K0$ then $\sum\limits_{z=0}^{K-1} q^{az^2} \equiv_{\Phi_K}\left( \frac{(a)_K}{K} \right)\G$ where $\left( \frac{r}{K} \right)$ is the Legendre symbol.
\end{lem}
\begin{proof} 1. By definition,
\[\G\tilde\G
=\left(\sum\limits_{x=0}^{K-1}q^{x^2}\right)\left(\sum\limits_{y=0}^{K-1}q^{-y^2}\right)
=\sum\limits_{x,y=0}^{K-1}q^{x^2-y^2}
\]
Change variables to $z=(x-y)_K$, $w=(x+y)_K$. Since $K$ is an odd prime, the map $(x,y)\mapsto(z,w)$ is a bijection from $\{0,1,\ldots,K-1\}^2$ to itself with inverse $x=\left(\frac{z+w}{2}\right)_K$, $y=\left(\frac{w-z}{2}\right)_K$. By Lemma \ref{modKmodphiK},
\[\G\tilde\G\ePK\sum\limits_{z,w=0}^{K-1}q^{zw}\ePK\sum\limits_{w=0}^{K-1}K\delta_{w=0}=K\]
where in the second step we perform the sum over $z$ using Lemma \ref{qazsum}.

\noindent 2. If $a\eK0$ then $az^2\eK0$ and so by lemma \ref{gammaKlemma}, $q^{az^2}\ePK1$ for all $z=0,\ldots,K-1$ so that $\sum\limits_{z=0}^{K-1}q^{za}\ePK{}K$.

\noindent If $(a)_K$ is a non-zero perfect square, say $(a)_K=(x^2)_K$ some $x\in\{1,\ldots,K-1\}$, then $az^2\eK(xz)_K^2$ while $z\mapsto(xz)_K$ is a bijection from $\{0,1,\ldots,K-1\}$ to itself (whose inverse is $u\mapsto\left(\frac{u}{x}\right)_K$) and so by Lemma \ref{modKmodphiK},
\[\G_a=\sum\limits_{z=0}^{K-1}q^{az^2}\ePK\sum\limits_{z=0}^{K-1}q^{(xz)_K^2}=\sum\limits_{z=0}^{K-1}q^{z^2}=\G
\]

\noindent Finally, if $(a)_K$ is not a perfect square then as $z$ ranges over $1,\ldots,K-1$, the number $(az^2)_K$ ranges over the complement of the set $X$ of remainders of perfect squares mod $K$, each being covered twice
\[
\big\{(az^2)_K\bigm|z=1,\ldots,K-1\big\}=\big\{0,1,\ldots,K-1\big\}-\big\{x^2\bigm|x=0,1,\ldots,K-1\big\}
\]
so that using Lemma \ref{modKmodphiK} and Lemma \ref{qazsum}
\[\G_a\ePK1+\sum\limits_{z=1}^{K-1}q^{(az^2)_K}
\ePK1+2\sum\limits_{\genfrac{}{}{0pt}{2}{u=0}{u\notin{}X}}^{K-1}q^u
=1+2\sum\limits_{u=0}^{K-1}q^u-2\sum\limits_{\genfrac{}{}{0pt}{2}{u=0}{u\in{}X}}^{K-1}q^u
=1+2\cdot0-2-2\sum\limits_{\genfrac{}{}{0pt}{2}{u=1}{u\in{}X}}^{K-1}q^u=-\sum\limits_{z=0}^{K-1}q^{(z^2)_K}
\ePK-\G
\]
\end{proof}
\begin{lem}\label{easyGauss} Let $a,b,c\in\ZK$. Then the generalized quadratic Gauss sum satisfies
\begin{align*}
\sum_{z=0}^{K-1} q^{az^2 + bz + c}\ePK
\begin{cases}
       q^{c-\frac{b^2}{4a}}\left(\frac{(a)_K}{K} \right) \G &\hbox{if $a\not\equiv_K0$}\\[1.5ex]
       0  &\hbox{if $a\eK0$, $b\not\equiv_K0$}\\[1.5ex]
       Kq^c  &\hbox{if $a\eK0$ and $b\eK0$}
    \end{cases}\>.
\end{align*}
\end{lem}
\begin{proof} If $a\eK0$ then $q^{az^2+bz+c}\ePK{}q^{bz+c}$ and the sum is given by Lemma \ref{qazsum}. Suppose now that $a\not\equiv_K0$. Completing the square, $az^2+bz+c=a\left(z+\frac{b}{2a}\right)^2+\delta$ where $\delta=c-\frac{b^2}{4a}$. Since $a\neK0$, thus $\frac{b}{2a},\delta\in\ZK$.  By Lemma \ref{modKmodphiK},
\[
\sum_{z=0}^{K-1}q^{az^2+bz+c}\ePK\sum_{z=0}^{K-1}q^{(az^2+bz+c)_K}
\ePK\sum_{z=0}^{K-1}q^{(a)_K\left(z+\left(\frac{b}{2a}\right)_K\right)^2+\left(\delta\right)_K}
=q^{\left(\delta\right)_K}\sum_{z=0}^{K-1}q^{(a)_K\left(z+\left(\frac{b}{2a}\right)_K\right)^2}
\]
As $z$ ranges over $0,\ldots,K-1$, its shift by the integer $\left(\frac{b}{2a}\right)_K$ ranges over all classes mod $K$, so that the last summation is congruent modulo $\Phi_K$ to $\G_a$ whose value is given by Lemma \ref{generalizedquadraticgausssum}, and the result follows.
\end{proof}

\begin{exmp}
For $K=3$, $\G=1+q+q^4$, $\tilde\G=1+q^{-1}+q^{-4}$, $\Phi_3=1+q+q^2$ and so
    \[\G\tilde\G=3+(q^2-1+2q^{-1}-q^{-2}+q^{-4})\Phi_3
    \]
\end{exmp}

\subsection{\textsc{Divisibility Properties}}
Since $\ZK$ is an integral domain thus the ring $\ZKh$ is an integral domain. The invertibles of $\ZKh$ are formal power series whose constant terms are not divisible by $K$ (that is, whose constant terms are invertible in $\ZK$). In particular, $K,\G,\tilde\G,\Phi_K\in\ZKh$ all have constant term $K$ and so are not invertible.

\begin{lem}\label{KGPhiirreducible} The elements $K, \G, \tilde\G, \Phi_K$ are \textbf{irreducible} in $\ZKh$.
\end{lem}
\begin{proof} Let $x$ be one of $K,\G,\tilde\G,\Phi_K\in\ZKh$. Suppose otherwise that $x$ is not irreducible, that is, it can be written as $x = ab$ where $a,b$ are both not units. Then both $a_0, b_0  \in K \ZK$ which implies that $a_0 b_0 \in K^2 \ZK $, however the constant term in $x$ is $K\not\in{}K^2\ZK$, a contradiction.
\end{proof}

\begin{prop} The ideal $<K,h>$ is \textbf{the unique maximal ideal} in $\ZKh$, that is any proper ideal in $\ZKh$ is contained in $<K,h>$.
\end{prop}
\begin{proof} Let $S$ be the set of all non-invertible elements in $\ZKh$, we will prove that $S = <K,h>$.\\
$(<K,h> \subseteq S)$  Suppose that $a(h)\in <K,h>$, so that $a(h)= Kb(h)+ hc(h)$ for some $b(h),c(h)\in\ZKh$. Since $a_0=Kb_0\eK0$, thus $a(h)$ is not invertible, that is $a(h)\in{}S$.\\
 $(S\subseteq <K,h>)$ Suppose that $a(h)=\sum_{r=0}^\infty{}a_rh^r\in{}S$, that is, it is a non-invertible element of $\ZKh$. Then $a_0\eK0$. Thus $a(h)=K\cdot\frac{a_0}{K}+h\sum_{r=0}^\infty{}a_{r+1}h^r\in<K,h>$ as required.

\noindent Suppose $I$ is a proper ideal in $\ZKh$. Then $I\subset{}S$, for if not, then $I$ would contain an element $x\in{}I\backslash{}S$. By the above, $x$ is invertible and so since $I$ is an ideal, it would contain all of $\ZKh$, a contradiction.
\end{proof}

\noindent For an odd prime $K$, the ring $\ZK$ is the localization of $\Z$ at the prime ideal $K\Z$. Since $<K>$ is the unique maximal ideal in $\ZK$, thus $\ZK$  is a Noetherian regular local ring. The ring of formal power series $\ZKh$ itself is also a Noetherian local ring with maximal ideal $<K,h>$. The Auslander–Buchsbaum theorem states that every regular local ring is a unique factorization domain and hence $\ZKh$ is a UFD. It follows from Lemma \ref{KGPhiirreducible} that the ideals generated by singleton elements $K, \Phi_K, \G, \tilde\G$ in $\ZKh$ are prime in $\ZKh$. Since the quotient of a commutative ring by a prime ideal is an integral domain, the following lemma follows, which we will use in \S3.5.

\begin{lem}\label{ZKhisaUFD}
The quotient ring $\ZKh/ < \Phi_K > $ is an integral domain.
\end{lem}

\subsection{\textsc{Squaks}}

\begin{defn}\label{defsquak}
    By a \textbf{squak} is meant a map $\sA$ from the set of primes (possibly omitting a finite number) to $\Q[[h]]$ such that $\sA(K)\in\ZKh$ for all sufficiently large prime $K$.
\end{defn}

\noindent Squaks form a ring under addition and multiplication of power series. For example, $K$, $\G$, $\tilde\G$, $\Phi_K$ all define squaks since they lie in $\Zh$. The involution $\psi$ defined in the previous subsection by $q\mapsto{}q^{-1}$ acts on squaks, preserves trivial squaks, equivalence of squaks and the properties of being independent of $K$ and almost independent of $K$. Since $\psi(\Phi_K)=q^{1-K}\Phi_K$ and $q^{1-K}$ is invertible in $\Zh$, thus $\psi$ preserves $K$-equivalence.

\begin{defn}
    A squak $\sA$ is said to be \textbf{independent of $K$} if there exists $a(h)\in\Qh$ such that $\sA(K)=a(h)$ for all sufficiently large prime $K$. A squak $\sA$ is said to be \textbf{trivial} if $\sA(K)=0$ for all sufficiently large prime $K$.   Two squaks are said to be \textbf{equivalent} if their difference is trivial.
\end{defn}

\begin{defn}\label{almostindependentdefn} A squak $\sA$ is said to be \textbf{almost independent of $K$}  if there exist integers $\cN,K_0\in\N$ such that for all prime $K_1,K_2>K_0$,
\[K_1\equiv_\cN{}K_2 \hspace{5pt} \Rightarrow \hspace{5pt} \sA(K_1)=\sA(K_2)\>.\]
\end{defn}

\noindent Thus a squak $\sA$ which is almost independent of $K$ is specified by a finite set of power series $a^{(i)}(h)\in\Qh$ for $1<i<\cN$ coprime to $\cN$ and then
\[\sA(K)=a^{(i)}(h)\hbox{ for all sufficiently large prime $K\equiv_\cN{}i$}\]
The minimal such integer $\cN$ will be called the \textbf{characteristic} of the squak.

\begin{lem}\label{Legendresymbolalmostindep}
    Suppose $a\in\Q^\times$. Then $\left(\frac{(a)_K}{K} \right)$ is a squak almost independent of $K$.
\end{lem}
\begin{proof}
    Set $a=\frac{m}{n}$ in reduced form, $m\in\Z^\times$, $b\in\N$ coprime. Write $mn=2^kl$ with $k\in\Zpos$ and $l\in\Z^\times$ odd. For odd $K$ coprime to $m$ and $n$, $an^2=mn$ and so $(a)_Kn^2\eK{}mn$ so that
    \[
\left(\frac{(a)_K}{K} \right)=\left(\frac{mn}{K} \right)=\left(\frac{2}{K} \right)^k\cdot\left(\frac{l}{K} \right)    \]
Recall that $\left(\frac{2}{K} \right)=(-1)^{\frac{K^2-1}{8}}$ depends only on $K$ modulo 8. Meanwhile by quadratic reciprocity, $\left(\frac{l}{K} \right)=(-1)^{\frac{K-1}{2}\cdot\frac{l-1}{2}}\left(\frac{K}{l} \right)$. Since $(-1)^{\frac{K-1}{2}}$ depends only on $K$ modulo 4 and $\left(\frac{K}{l} \right)$ depends only on $K$ modulo $l$, the result follows.
\end{proof}

\begin{defn}
 Two squaks $\sA$ and $\sB$ are said to be $K$-equivalent if for all sufficiently large prime $K$, $\sA(K)\ePK\sB(K)$.
\end{defn}

\noindent It is immediate that $K$-equivalence of squaks defines an equivalence relation.

\begin{defn}\label{almostindependentmodPhiKdef} A squak $\sA$ is said to be \textbf{almost independent of $K$ mod $\Phi_K$} if there exists a squak $\sB$ which is almost independent of $K$ and is $K$-equivalent to $\sA$.
\end{defn}

\begin{exmp} Suppose $c\in\Q^\times$, written in reduced form as $c=\frac{c_1}{c_2}$ with $c_1\in\Z$, $c_2\in\N$ coprime.
\begin{itemize}
    \item  $\sA=q^c$ defines a squak which is independent of $K$, with $\sA(K)\in\ZKh$ for $K$ coprime to  $c_2$.
    \item  $\sB=q^{\{Kc\}}$ defines a squak almost independent of $K$ with characteristic $\cN=c_2$ since the fractional part $\{cK\}$ depends only on $K$ modulo $c_2$.
    \item  $\sC=q^{-\lfloor{Kc}\rfloor}$ is not almost independent of $K$ but is almost independent of $K$ mod $\Phi_K$ since $\sB$ and $\sC$ are $K$-equivalent squaks because $\{Kc\}=Kc-\lfloor{Kc}\rfloor\eK-\lfloor{Kc}\rfloor$ for $(K,c_2)=1$, so that by Lemma \ref{modKmodphiK}, $q^{\{Kc\}}\ePK{}q^{-\lfloor{Kc}\rfloor}$.
\end{itemize}
\end{exmp}

\begin{prop}\label{almostindepKequivtrivial} A squak which is almost independent of $K$ and $K$-equivalent to a trivial squak is trivial.
\end{prop}
\begin{proof}  Suppose that $\sA$ is almost independent of $K$ and $K$-equivalent to a trivial squak. By definition, there exists $K_0$ so that for all prime $K>K_0$, $\sA(K)\ePK0$ while $\sA(K)$ depends only on $K$ (mod $\cN$) for some characteristic $\cN\in\N$. Thus there exists another squak $\sB$ so that $\sA(K)=\sB(K)\cdot\Phi_K$ for all prime $K>K_0$, while for primes $K_1,K_2>K_0$,
\[K_1\equiv_\cN{}K_2 \implies \sA(K_1)=\sA(K_2)\]
We need to prove that $\sA(K)=0$ for all sufficiently large $K$. Suppose that $K>\max(K_0,\cN)$ is prime. For any prime $K'\equiv_\cN{}K$ with $K'>K_0$,
\[\sA(K)=\sA(K')\in\ZKph\]
Since $\Phi_{K'}\equiv_{K'}h^{K'-1}$,
\[\sA(K)=\sB(K')\Phi_{K'}\equiv_{K'}\sB(K')h^{K'-1}\]
so that the coefficient of $h^r$ in $\sA(K)$ is divisible by $K'$ for all prime $K'$ for which $r<K'-1$, $K'\equiv_\cN{}K$ and $K'>K_0$. For fixed $r$, by Dirichlet's theorem, there are infinitely many such primes $K'$ and so all the coefficients of $\sA(K)$ must vanish, that is $\sA$ is a trivial squak.
    \end{proof}

\begin{cor}\label{uniqueuptophiK} If $\sA$ is a squak almost independent of $K$ mod $\Phi_K$ then there exists a \textbf{\textit{unique squak}} $\sB$ up to equivalence which is almost independent of $K$ and such that $\sA$, $\sB$ are $K$-equivalent.
\end{cor}
\begin{proof} By definition of almost independence mod $\Phi_K$, $\sB$ exists. To prove uniqueness, suppose that $\sC$ is another squak almost independent of $K$ such that $\sA$ and $\sC$ are $K$-equivalent. By transitivity of $K$-equivalence, $\sB$ and $\sC$ are $K$-equivalent, so that $\sB-\sC$ is a squak $K$-equivalent to the trivial squak. Additionally $\sB-\sC$ is almost independent of $K$ and so by Proposition \ref{almostindepKequivtrivial}, $\sB-\sC$ is a trivial squak, that is $\sB$ and $\sC$ are equivalent squaks.
\end{proof}

\begin{lem}\label{Kmodphiisnotalmostind} There does not exist a squak almost independent of $K$ which is $K$-equivalent to the squak $K$ (whose value on a prime $K$ is the formal power series with only the constant term $K$).
\end{lem}
\begin{proof} Suppose otherwise that $\sA$ is a squak almost independent of $K$ which is $K$-equivalent to $K$. Then there exist integers $K_0$ and $\cN$ so that $\sA(K)\ePK{}K$ for all primes $K>K_0$ while $\sA(K)=\sA(K')$ for all primes $K,K'>K_0$ congruent mod $\cN$. Then
\[\sA(K)=K+\sB(K)\Phi_K\]
for some squak $\sB$. Since $\Phi_{K'}\equiv_{K'}h^{K'-1}$, thus for all primes $K,K'>K_0$ congruent mod $\cN$,
\[\sA(K)=\sA(K')=K'+\sB(K')\Phi_{K'}\equiv_{K'}\sB(K')h^{K'-1}
\]
which means that the coefficient of $h^r$ in $\sA(K)$ is divisible by $K'$ for all primes $K'>K_0$ for which $r<K'-1$ and $K'\equiv_\cN{}K$. For any fixed $r$,  this is true for infinitely many primes $K'$ and so $\sA(K)=0$ for all prime $K>K_0$, that is $\sA$ is a trivial squak. However this is a contradiction since $K$ is not $K$-equivalent to the trivial squak as $K$ is not divisible by $\Phi_K$ in $\ZKh$.
\end{proof}

\begin{prop}\label{noalmostindepG} There does not exist a squak almost independent of $K$ which is $K$-equivalent to the Gauss sum squak $\G$.
\end{prop}
\begin{proof} Suppose otherwise that $\sA$ is a squak almost independent of $K$ which is $K$-equivalent to $\G$. Then $\psi(\sA)$ is a squak almost independent of $K$ which is $K$-equivalent to $\psi(\G)=\tilde\G$ and so their product $\sA\psi(\sA)$ is a squak almost independent of $K$ which is $K$-equivalent to $\G\tilde\G$. By Lemma \ref{easyGauss}, $\G\tilde\G$ is $K$-equivalent to $K$, which gives a contradiction by Proposition \ref{Kmodphiisnotalmostind}.
\end{proof}

\begin{lem}\label{GPhilinearcomblemma} The only squak almost independent of $K$ which is equivalent to a linear combination of  $\Phi_K$ and $\G$ is trivial.
\end{lem}
\begin{proof} Suppose that $\sA$ is a squak almost independent of $K$ and that
\[\sA(K)=\sB(K)\cdot\Phi_K+\sC(K)\cdot\G\]
for some squaks $\sB,\sC$ and all sufficiently large prime $K$. Then there exist integers $K_0$ and $\cN$ so that $\sA(K)=\sA(K')$ for all primes $K,K'>K_0$ congruent mod $\cN$. Fix $K>K_0$ prime. For all primes $K'\equiv_\cN{}K$ with $K'>K_0$,
\[
\sA(K)=\sA(K')=\sB(K')\cdot\Phi_{K'}+\sC(K')\cdot\G'\ePKp\sC(K')\cdot\G'
\]
where $\mathbb{G}'$ denotes the Gauss sum at $K'$. Applying the involution $\psi$,
\[
\psi(\sA(K))\ePKp\psi(\sC(K'))\cdot\tilde\G'
\]
and the product gives
\[\sA(K)\psi(\sA(K))\ePKp\sC(K')\psi(\sC(K'))\G'\tilde\G'\ePKp{}K'\sC(K')\psi(\sC(K'))\]
That is, for some squak $\sD$,
\[\sA(K)\psi(\sA(K))=K'\sC(K')\psi(\sC(K'))+\Phi_{K'}\cdot\sD(K')\]
Since $\Phi_{K'}\equiv_{K'}h^{K'-1}$, thus $\sA(K)\psi(\sA(K))\equiv_{K'}h^{K'-1}\sD(K')$ and so the coefficient of $h^r$ in $\sA(K)\psi(\sA(K))$ is divisible by $K'$ for all primes $K'>K_0$ for which $r<K'-1$ and $K'\equiv_\cN{}K$. For any fixed $r$,  this is true for infinitely many primes $K'$ and so $\sA(K)\psi(\sA(K))=0$ for all primes $K>K_0$. Suppose otherwise that $\sA$ is not a trivial squak. Then there exists $K>K_0$ prime for which $\sA(K)\not=0$, say its first non-zero coefficient in the power series is of $h^m$. Then the coefficient of $h^m$ in $\psi(\sA(K))$ will also be non-zero and so the coefficient of $h^{2m}$ in the product will be non-zero, a contradiction.
\end{proof}

\begin{lem}\label{almostindependentmodK} A squak which is almost independent of $K$ and is divisible by $K$ must be trivial.
\end{lem}
\begin{proof} This follows immediately from Lemma \ref{GPhilinearcomblemma} since $K\in<\Phi_K,\G>$ because $\G\tilde\G\ePK{}K$.
\end{proof}

 \noindent Note: Proposition \ref{noalmostindepG} is a special case of Lemma \ref{GPhilinearcomblemma}, while Lemma \ref{Kmodphiisnotalmostind} is a special case of Lemma \ref{almostindependentmodK}.

\subsection{{\textsc{$\G$ decompositions}}}

\begin{defn}\label{Gdecomposition}
A squak $\sA$ is said to admit a $\G$-decomposition if there exists $d\in\N$ and squaks $\sB_0,\ldots,\sB_d$  which are almost independent of $K$ such that $\sA$ is $K$-equivalent to $\sum\limits_{r=0}^d\sB_r\G^r$. The minimal value of $d$ is called the level of the $\G$-decomposition.
\end{defn}

\begin{lem}\label{nonontrivalGdecompofzero} The only $\G$-decomposition of a trivial squak is trivial.
\end{lem}
\begin{proof}
Suppose that $\sB_0,\ldots,\sB_d$ are squaks which are almost independent of $K$ and such that $\sum\limits_{r=0}^d\sB_r\G^r$ is $K$-equivalent to the trivial squak. To obtain a contradiction, suppose otherwise that $\sB_r$ are not all trivial; let $m$ denote the smallest integer for which $\sB_m$ is non-trivial. Then for all sufficiently large primes $K$,
\[
\G^m\left(\sum\limits_{r=m}^d\sB_r(K)\G^{r-m}\right)=\sum\limits_{r=0}^d\sB_r(K)\G^r\ePK0
\]
By Lemma \ref{ZKhisaUFD},  $\ZKh/<\Phi_K>$ has no zero divisors and so
$\sum\limits_{r=m}^d\sB_r(K)\G^{r-m}\ePK0$.
Equivalently, for some squak $\sC$, then $\sum\limits_{r=m}^d\sB_r(K)\G^{r-m}=\sC(K)\Phi_K$, which can be rearranged as
\[
\sB_m(K)=\sC(K)\Phi_K-\G\sum\limits_{r=m+1}^d\sB_r(K)\G^{r-m-1}
\]
for all sufficiently large primes $K$. This expresses the squak $\sB_m$, which is almost independent of $K$, as equivalent to a linear combination of $\G$ and $\Phi_K$. By Lemma \ref{GPhilinearcomblemma}, this implies that $\sB_m$ is trivial, a contradiction.\end{proof}

\begin{thm}\label{uniquenessinarbitrarydimension}\ If a squak admits a $\G$-decomposition then it is unique up to equivalence of squaks.
\end{thm}
\begin{proof} By taking the difference of two $\G$-decompositions, the result immediately follows from Lemma \ref{nonontrivalGdecompofzero}.
\end{proof}

\section{\textsc{Bernoulli Polynomials \& Symbolic Integration}}
In this section we introduce a symbolic notation where a power of the formal symbol $B$ is interpreted as a Bernoulli number, as well as integration whose limits are of the form $a+B$ and double integration over polygonal regions, again shifted by a formal symbol $B$. Similar notation (without integration) appears in the literature under the name \textbf{umbral calculus}.

\subsection{\textsc{Bernoulli numbers and symbolic notation  }}
\begin{defn} The Bernoulli numbers $B_n$ are defined by their generating function
\begin{align}\frac{t}{e^t-1}=\sum\limits_{n=0}^\infty\frac{B_nt^n}{n!}\label{Bnumgenfn}
\end{align}

\end{defn}
\noindent The first few Bernoulli numbers are $B_0=1$, $B_1=-\frac12$, $B_2=\frac16$. Replacing $t$ by $-t$ in the generating function,
\[\sum\limits_{n=0}^\infty\frac{(-1)^nB_nt^n}{n!}=\frac{t}{1-e^{-t}}=t+\frac{t}{e^t-1}=t+\sum\limits_{n=0}^\infty\frac{B_nt^n}{n!}
\]
from which it follows that $(-1)^nB_n=B_n$ for $n>1$, that is, odd Bernoulli numbers vanish after $B_1$.

We choose to use a \textbf{symbolic notation} in which polynomials in the symbol $B$ are evaluated by replacing $B^n$ by $B_n$ and extending linearly. This extends to polynomials in other variables, whose coefficients are polynomials in $B$. A power series in $h$ whose coefficients are polynomials in $B$ is similarly evaluated to give a power series in $h$ whose coefficients are real numbers. Since $B_mB_n\not=B_{m+n}$ this map is not a ring homomorphism. Similarly, if $p(x)\in\R[x]$ and $a(B),b(B)\in\R[B]$ then  $\int\limits_{a(B)}^{b(B)}p(x)dx$ will be understood by first evaluating the integral and obtaining a polynomial in $B$ and only then evaluating powers of $B$ by replacement with Bernoulli numbers to give a numerical value.
\begin{defn} Let $\p(x) \in \Rxh$ be a formal power series in $h$ whose coefficients are polynomials in $x$, say $\p(x) = \sum_{n=0}^{\infty} p_n (x) h^n  $ and $a(B),b(B)\in\R[B]$. Then we define the symbolic integral $\int_{a(B) }^{b(B)}\p(x) dx$ by term-by-term integration,
\begin{align*}
    \int_{a(B)}^{b(B)} \p(x) dx := \sum_{n=0}^\infty  \left(\int_{a(B)}^{b(B)} p_n(x) dx \right) h^n\>.
\end{align*}
\end{defn}
\noindent Always the evaluation of powers of $B$ comes at the last stage of the calculation. Since the symbolic integral is defined as an evaluation of a real integral (with $B$ as a parameter), it is additive under subdivision of the range of integration.

\begin{defn}The Bernoulli polynomials $\{B_n (x)\}_{n=0}^\infty$ are defined through the generating function,
\begin{align}
        \frac{t e^{xt} }{e^t -1} \equiv \sum_{n=0}^\infty B_n (x) \frac{t^n }{n!}\>,\quad\hbox{for $|t|<2\pi$}\label{Bpolygenfn}
\end{align}
\end{defn}
\begin{lem}\label{Bpolyprop}
\begin{enumerate}
    \item The first few Bernoulli polynomials are $B_0(x)=1$, $B_1(x)=x-\frac12$, $B_2(x)=x^2-x+\frac16$.
    \item $B_n(0)=B_n$
    \item $B_n(1-x)=(-1)^nB_n(x)$. That is, even [odd] Bernoulli polynomials are symmetric [antisymmetric] around $x= \frac{1}{2}$ (see \cite{L40}).
    \item $B'_n(x)=nB_{n-1}(x)$ for $n>0$.
    \item $B_n(x)=(x+B)^n$.
\end{enumerate}
\end{lem}
\begin{proof}
\begin{enumerate}
\item Explicit expansion of $\frac{te^{xt}}{e^t-1}=(1+xt+\frac{x^2t^2}{2}+\cdots)(1+\frac{t}{2}+\frac{t^2}{6}+\cdots)^{-1}$.
\item After substituting $x=0$, \eqref{Bpolygenfn} coincides with \eqref{Bnumgenfn}.
\item Replacing $x$ by $1-x$ in \eqref{Bpolygenfn} gives
$\sum\limits_{n=0}^\infty B_n (1-x) \frac{t^n }{n!}=\frac{t e^{t-xt} }{e^t -1}=\frac{-te^{-xt}}{e^{-t}-1}
=\sum\limits_{n=0}^\infty B_n (x) \frac{(-t)^n }{n!}$
from which it follows by equating coefficients that $B_n(1-x)=(-1)^nB_n(x)$.
\item Differentiating \eqref{Bpolygenfn} partially with respect to $x$ gives
$\sum\limits_{n=0}^\infty B'_n (x) \frac{t^n }{n!}=\frac{t^2 e^{xt} }{e^t -1}=\sum\limits_{n=0}^\infty B_n (x) \frac{t^{n+1} }{n!}$
from which it follows by equating coefficients that $B'_n(x)=nB_{n-1}(x)$ for $n>0$.
\item Comparing \eqref{Bpolygenfn} with \eqref{Bnumgenfn},
$\sum\limits_{k=0}^\infty B_k (x) \frac{t^k }{k!}=\sum\limits_{n=0}^\infty\frac{B_nt^n}{n!}e^{xt}=\sum\limits_{m,n=0}^\infty\frac{B_nt^{n+m}x^m}{n!m!}$
from which it follows by equating coefficients that $B_k(x)=\sum\limits_{m=0}^k\binom{k}{m}B_{k-m}x^m=(x+B)^k$.
\end{enumerate}
\end{proof}

\noindent An important use of Bernoulli polynomials comes from the fact that they evaluate power summations (see \cite{C92}).
\[
    \sum_{x=0}^{N-1} x^m  =  1^m + 2^m +  ... + (N-1)^m = \frac{1}{m+1} \left( B_{m+1}(N)- B_{m+1} \right)
    = \int_{0}^{N} B_m (x) dx=\int_{B}^{N + B } x^m dx \>.
\]

\begin{lem}\label{EulerMaclaurin}  \textbf{(Euler-Maclaurin formula)} Let $f(x) \in \R[x]$ with anti-derivative
$F(x) \in \R[x]$. Then for all $a,b \in \Z$ with $a<b$
\begin{align*}
    \sum_{i=a}^{b-1} f(i) = \int_{a+ B }^{b+ B }f(x) dx = F(b+ B ) - F(a+ B ).
\end{align*}
\end{lem}
\noindent One advantage of working with symbolic integrals is that they (the middle term in Lemma \ref{EulerMaclaurin}) are well-defined for non-integer values of $a,b$ while the ordinary summation (left hand side) is not.

\begin{lem}\label{integralcombination} If $n\in\Zpos$ then $\int_B^{x+ B } \binom{z}{n} dz=\binom{x}{n+1}  $ for all $x\in\R$.
\end{lem}
\begin{proof} Since $\binom{z}{n}$ is a polynomial of degree $n$ in $z$ thus \[g(x):= \int_B ^{x+ B }\binom{z}{n} dz\] is a polynomial of degree $n+1$ in $x$. For $x\in\N$, by Lemma \ref{EulerMaclaurin},
$g(x)=  \sum_{z=0}^{x-1}\binom{z}{n }=\binom{x}{n+1}$ and so this identity holds for all $x\in\R$.
\end{proof}

\begin{lem} \label{BinteasyGauss} Suppose that $f(z)=az^2+bz+c\in\ZK[z]$ with $a\not\equiv_K0$. Then
   $\int_B ^{K+ B } q^{f(z )}dz\ePK\left( \frac{(a)_K}{K}  \right) q^{f_0} \G $
where $f_0$ is the value of $f(z)$ at its critical point.
\end{lem}
\begin{proof}
    By Lemma \ref{EulerMaclaurin}, $\int_B ^{K+ B } q^{f(z )}dz=\sum_{m=0}^{K-1}q^{f(m)}$ which gives the result by Lemma \ref{easyGauss}, since $f_0=c-\frac{b^2}{4a}$.
\end{proof}

\begin{lem}\label{Bintegralbelonging} If $\p(z;q) \in\PzZKh$ then $\int_B ^{z_0 + B } \p (w;q) dw \in\ZKh$ for all $z_0\in\ZK$. Equivalently, $\int_B ^{z+ B }\p(w;q) dw \in\PzZKh$.
\end{lem}
\begin{proof} Suppose that $\p(z;q) \in\P(z;\ZK)[[h]]$, so that $\p(z;q) = \sum_{n=0}^\infty p_n (z) h^n  $ with $p_n (z) \in \P(z;\ZK)$ for all $n$. That is, $p_n (z)= \sum_{r=0}^{\deg(p_n) }p_{n,r} \binom{z}{r}  $ for some $p_{n,r}\in\Z_{(K)}$, $0\leq{}r\leq\deg(p_n)$. By definition
\[\int_B ^{z + B } \p(w;q) dw = \sum_{n=0}^\infty h^n \int_B ^{z + B } p_n (w) dw
=\sum_{n=0}^\infty h^n  \sum_{r=0}^{\deg(p_n) }p_{n,r}\int_B ^{z + B }\binom{w}{r } dw
=\sum_{n=0}^\infty h^n\sum_{r=0}^{\deg(p_n) }p_{n,r}\binom{z_0}{r+1}
\] where in the last step we used Lemma \ref{integralcombination}. The RHS is now a power series whose coefficients in $\P(z;\ZK)$ as required. Equivalently by Lemma \ref{P(z,R)characterization}, for any $z_0\in\ZK$, the symbolic integral lies in $\ZKh$.

Alternatively, the lemma can be proved using Lemma \ref{pzqtoZKh} since the symbolic integral is an element of $\Q[z_0][[h]]$ and for integer $z_0$ is equal to a sum of elements in $\ZKh$ which thus lies in $\ZKh$.
\end{proof}

\noindent Subtracting two such integrals as in Lemma \ref{Bintegralbelonging} with the same integrand and different values of $z_0$, allows symbolic integration between $x_0+B$ and $y_0+B$ for any $x_0,y_0\in\ZK$. Combining with  Lemma \ref{qpowerfz}, this provides a source of squaks as well as the possibility of iterating the integration.

\begin{lem}\label{Bintegralgqpowerf} If $f(z),g(z)\in\PzZK$, then $\int_{ B}^{z+ B } q^{f(w)}g(w) dw\in\PzZKh$. In particular, for all $x_0,y_0\in\ZK$, $\int_{x_0+ B}^{y_0+ B } q^{f(w)}g(w) dw\in\ZKh$.
\end{lem}

\begin{lem}\label{Bintegralgqpowersquak} If $f(z),g(z)\in\Q[z]$ and $x_0,y_0\in\Q$ then $\int_{x_0+B}^{y_0+B}q^{f(w)}g(w)dw$ defines a squak.
\end{lem}

\noindent Furthermore, equivalence in $\P(z;\ZK)[[h]]$ modulo $\Phi_K$ is  preserved under symbolic integration.

\begin{lem} \label{Bintegralcongruence} If $\p(z;q)\equiv_{\Phi_K}\f(z;q)$ in $\PzZKh$, then \[\int_B ^{z+ B} \p(w;q) dw \equiv_{\Phi_K} \int_B ^{z+ B } \f(w;q) dw\] in $\PzZKh$, while for all $z_0\in\ZK$ it holds that $\int_B ^{z_0+ B} \p(w;q) dw \equiv_{\Phi_K} \int_B ^{z_0+ B } \f(w;q) dw$ in $\ZKh$.
\end{lem}
\begin{proof} By assumption,  $\p (z;q) = \f(z;q) + \c(z;q) \Phi_K$ for some  $\c (z;q) \in \P(z;\ZK)[[h]]$. By linearity, integrating symbolically gives,
\[
\int_B ^{z+ B } \p (w;q) dw
= \int_B ^{z+ B } \f(w;q) dw + \left( \int_B^{z+B} \c(w;q)dw \right)\Phi_K \equiv_{\Phi_K} \int_B ^{z+ B } \f(w;q) dw
\]
in $\P(z;\ZK)[[h]]$ by Lemma \ref{Bintegralbelonging}. On evaluation at any $z_0\in\ZK$, all the terms lie in $\ZKh$ and so the congruence also holds in $\ZKh$.
\end{proof}

\subsection{\textsc{Symbolic Iterated Integrals}}
An alternative way of defining the symbolic integral $\int_{a+B}^{b+B}p(x)dx$ of the previous subsection, is by using the following operator on polynomials.
\begin{defn} Define an $\R$-linear operator $\B : \R[x] \rightarrow \R [x]$ by $\B(x^m) := B_m (x) $ for $m\geq0$.
\end{defn}
\noindent In particular, $\B[C]=C$ for constants $C\in\R$. Since $B_m(x)=(x+B)^m$ in the notation of the previous section, thus for any polynomial $p(x)\in\R[x]$ then $\B(p(x))=p(x+B)$. This extends to $\Rxh$.
\begin{lem}\label{intaBbB} If $a,b \in \R$ and $p(x) \in \R[x]$ then $ \int_{a+ B }^{b+ B } p(x) dx = \int_a ^b \B[p(x)] dx $.
\end{lem}
\begin{proof} On monomials, $\int_{a+ B }^{b+ B }  x^m dx = \frac{1}{m+1} ((b+B)^{m+1}-(a+B)^{m+1})=\frac1{m+1}( B_{m+1} (b) - B_{ m+1} (a) )$ which is equal to $\int_a^bB_m(x)dx$ because $\frac{d}{dx} B_{m+1}(x) = (m+1) B_m(x)$ by Lemma \ref{Bpolyprop}(4).
\end{proof}
\begin{lem} \label{diffBcommute} Differentiation and the $\B$-operator commute, that is, $\frac{d}{dx} \B[ p(x)] = \B [p'(x)]$ for all $p(z) \in \R [z]$.
\end{lem}
\begin{proof} By linearity, it suffices to check for $p(x)=x^m$ where the fact follows from $\frac{d}{dx} B_m (x) = m B_{m-1}(x) $.
\end{proof}
\begin{lem}\label{intBcommute} On $\R[x]$, the indefinite integral (up to a constant) and $\B$ commute. For the definite integral, $ \int_{a+ B }^{b+ B } p(x) dx =\left( \B  \left[ \int p(x) dx \right]\right) \big|_{x=a}^{x=b}$ for any $p(x) \in \R[x]$ and constants $a,b \in \R$.
\end{lem}
\begin{proof}
    For the indefinite integral, the property follows immediately from Lemma \ref{diffBcommute}. It can also be explicitly checked:  $\B\left[\int x^m dx\right] = \B \left[\frac{x^{m+1}}{m+1}+C\right] = \frac{1}{m+1} B_{m+1} (x)+C$ while $\int \B[x^m ] dx =  \int B_m (x) dx = \frac{1}{m+1}B_{m+1}(x)+C$, since $\frac{d}{dx}B_{m+1}(x)=(m+1)B_m(x)$, so that the result holds on monomials and thus by linearity also on arbitrary polynomials. The second half on definite integrals follows immediately by combining with Lemma \ref{intaBbB}.
\end{proof}

\begin{lem}\label{symbolicinttranslationinversion}{\bf (Translation and central inversion of symbolic integrals)}
    Suppose that $f(x)\in\R[x]$, $c\in\R$. Then
    \begin{align*}
    \int_{a+ B }^{b+ B } f(x+c) dx &= \int_{a+c+B}^{b+c+B} f(x) dx\\
    \int_{1-b+B}^{1-a+B}f(-x)dx&=\int_{a+B}^{b+B} f(x) dx
\end{align*}
\end{lem}
\begin{proof}
    Pick a primitive $F$ for $f$. Then $F(x+c)$ is a primitive for $f(x+c)$ while $-F(-x)$ is a primitive for $f(-x)$ and so
  \begin{align*}
    \int_{a+ B }^{b+ B } f(x+c) dx &= F(b+c+B)-F(a+c+B)=\int_{a+c+B}^{b+c+B} f(x) dx\\
    \int_{1-b+B}^{1-a+B}f(-x)dx&=(-F(a-1-B))-(-F(b-1-B))=F(b+B)-F(a+B)=\int_{a+B}^{b+B} f(x) dx
\end{align*}
The second equality in the second line follows from Lemma \ref{Bpolyprop}(3),(5) which gives that $(1-x+B)^n=B_n(1-x)=(-1)^nB_n(x)=(-1)^n(x+B)^n=(-x-B)^n$ and so for any polynomial $F$, $F(1-x+B)=F(-x-B)$.
\end{proof}

\medskip\noindent \textbf{More than one variable:} On polynomials in more than one variable, different $\B$ operators can be defined for each variable. Thus $\B_x$ and $\B_y$ will be two commuting operators on $\R[x,y]$ defined by
\[\B_x(x^my^n):=B_m(x)y^n,\qquad\B_y(x^my^n):=x^mB_n(y)\]

\begin{defn} (\textbf{Symbolic double integral}) For $g(x),h(x)\in\R[x]$, define the symbolic repeated integral $\int_{x_1+B}^{x_2+B}\int_{g(x)+B}^{h(x)+B}p(x,y)dydx$ for $x_1,x_2\in\R$, polynomials $g(x),h(x)\in\R[x]$ and $p(x,y)\in\R[x,y]$ to be $\R$-linear in $p$ with
\begin{align*}\int_{x_1+B}^{x_2+B}\int_{g(x)+B}^{h(x)+B}x^my^ndydx
&:=\int_{x_1+B}^{x_2+B}\frac{x^m}{n+1}\big(B_{n+1}(h(x))-B_{n+1}(g(x))\big)dx\\
&=\int_{x_1+B}^{x_2+B}\left(\int_{g(x)}^{h(x)}\B_y(x^my^n)dy\right)dx
=\int_{x_1}^{x_2}\B_x\left[\left(\int_{g(x)}^{h(x)}\B_y(x^my^n)dy\right)\right]dx
\end{align*}\end{defn}
\noindent On a rectangular domain,
\[
\int_{x_1+B}^{x_2+B}\int_{y_1+B}^{y_2+B}x^my^ndydx
=\int_{x_1}^{x_2}\B_x\left[\int_{y_1}^{y_2}\B_y[x^my^n]dy\right]dx
=\int_{x_1}^{x_2}\left(\int_{y_1}^{y_2}\B_x[\B_y[x^my^n]]dy\right)dx
\]
and so for all polynomials $p(x,y)$,
\[
\int_{x_1+B}^{x_2+B}\int_{y_1+B}^{y_2+B}p(x,y)dydx
=\int_{x_1}^{x_2}\int_{y_1}^{y_2}\B_x\B_y[p(x,y)]dydx
\]
where the order of integration can be reversed.

\noindent
Applying \ref{EulerMaclaurin} twice, we see that if $x_1,x_2\in\Z$ with $x_1<x_2$ and $g(x),h(x)\in\P(x;\Z)$ with $g(x)<h(x)$ for $x\in\Z$ with $x_1\leq{}x<x_2$, then
\[\int_{x_1+B}^{x_2+B}\int_{g(x)+B}^{h(x)+B}p(x,y)dydx=\sum_{m=x_1}^{x_2-1}\sum_{n=g(m)}^{h(m)-1}p(m,n)
\]

\begin{thm}\label{2dBintegralbelonging} Suppose that $g(z) \in \P(z;\ZK)$, $w_0,z_0 \in\ZK$ and $\p(z,w;q) \in \PzwZKh$. Then
\[\int_B ^{z_0 + B }\int_B ^{w_0 - g(z) + B  } \p(z,w;q) dw dz\in\ZKh\>.\]
\end{thm}
\begin{proof} Write $\p(z,w;q) = \sum_{n=0}^\infty  h^n p_n(z,w)$ for some $p_n(z,w)\in\PzwZK$. By definition, the iterated integral is
\[
F(z_0,w_0):=\int_{B}^{z_0 +B } \int_{ B}^{w_0 -g(z) +B  }\p(z,w;q) dw dz
=\sum\limits_{n=0}^\infty{}c_n(z_0,w_0)h^n\]
where $c_n(z_0,w_0)=\int_B ^{z_0 + B } \int_B ^{w_0 -g(z) + B } p_n(z,w)  dw dz$.
We need to show that $c_n(z_0,w_0)\in\ZK$ for all $z_0,w_0\in\ZK$ and $n\in\Z^{\geq0}$.
Since $p_n(z,w)\in\PzwZK$, by definition $p_n(z,w)=\sum_{k,l=0}^{N_n}a_{k,l,n}\binom{z}{k}\binom{w}{l}$ for some $N_n\in\N$ and $a_{k,l,n}\in\ZK$ for $0\leq{}k,l\leq{}N_n$. Now
\[c_n(z_0,w_0)=\sum_{k,l=0}^{N_n}a_{k,l,n}\int_B ^{z_0 + B } \int_B ^{w_0 -g(z) + B } \binom{z}{k}\binom{w}{l}  dw dz
=\sum_{k,l=0}^{N_n}a_{k,l,n}\int_B ^{z_0 + B }\binom{z}{k}\binom{w_0-g(z)}{l+1} dz
\]
where the inner symbolic integral was evaluated by Lemma \ref{integralcombination}. For fixed $k,l$, the integrand in the expression on the RHS takes values in $\ZK$ for all $z\in\Z$ because $w_0\in\ZK$ and $g(z)\in\ZK$ along with Lemma \ref{combinationbelongs}. This integrand is a polynomial in $z$ and hence by Lemma \ref{P(z,R)characterization} lies in $\P(z;\ZK)$, meaning that
\[
\binom{z}{k}\binom{w_0-g(z)}{l+1}=\sum\limits_{r=0}^{M_{k,l}}b_{k,l,r}\binom{z}{r}
\]
for some $M_{k,l}\in\N$ and coefficients $b_{k,l,r}\in\ZK$. Substituting this in the previous expression for $c_n(z_0,w_0)$ we get
\[
c_n(z_0,w_0)=\sum_{k,l=0}^{N_n}a_{k,l,n}\sum_{r=0}^{M_{k,l}}b_{k,l,r}\int_B ^{z_0 + B }\binom{z}{r} dz
=\sum_{k,l=0}^{N_n}a_{k,l,n}\sum_{r=0}^{M_{k,l}}b_{k,l,r}\binom{z_0}{r+1}
\]
where the last symbolic integral evaluation was by Lemma \ref{integralcombination}. Finally since $z_0,a_{k,l,n},b_{k,l,r}\in\ZK$ along with Lemma \ref{combinationbelongs}, we conclude that $c_n(z_0,w_0)\in\ZK$ as required.
\end{proof}

\begin{lem}\label{BBintegralsquakgqpowerf}
If $z_0,w_0\in\Q$, $g(z)\in\Q[z]$ and $p(z,w),h(z,w)\in\Q[z,w]$ then
\[\int_B ^{z_0 + B }\int_B ^{w_0 - g(z) + B  } h(z,w)q^{p(z,w)} dw dz\]
defines a squak.
\end{lem}

\subsection{\textsc{The Carry-Over Theorems}}
\begin{thm}\label{carry-overlemma2}
Suppose that $\p(z;q)\in\PzZKh$ and put $\c(z;q)=\p(z+ K;q)-\p(z;q)$. If $f(z)\in\ZKh$ and $a\in\ZK$ then in $\ZKh$
\[
\int_{a+B}^{a+K+B} q^{f(z)} \p(z;q) dz \ePK\int_B^{K+B} q^{f(z)} \p(z;q)dz +\int_B^{a+B}q^{f(z )}\c(z;q) dz
 \]
\end{thm}

\begin{proof}
Since $f(z)\in\ZK[z]$ then by Lemma \ref{qpowerfz}, $q^{f(z)}\in\PzZKh$ so that the integrands lie in $\PzZKh$ and thus by Lemma \ref{Bintegralbelonging}, all the integrals lie in $\ZKh$. Splitting the range of integration,
\[
\int_{a+B}^{a+K+B}q^{f(z)}\p(z;q)dz
=\int_{B}^{K+B}q^{f(z)}\p(z;q)dz+\int_{K+B}^{a+K+B}q^{f(z)}\p(z;q)dz-\int_{B}^{a+B}q^{f(z)}\p(z;q)dz
\]
Since $f(z)\in\ZK[z]$, thus for all $z_0\in\Z$, $f(z_0+K)\eK{}f(z_0)$ and so by Lemma \ref{modKmodphiK}, $q^{f(z_0+K)}\ePK{}q^{f(z_0)}$ in $\ZKh$. By Lemma \ref{pzqmodPhicharacterisation}, $q^{f(z+K)}\ePK{}q^{f(z)}$ in $\PzZKh$. Since $\p(z+ K;q)=\ p(z;q) + \c(z;q)$, thus
\[ q^{f(z+K)}\p(z+ K;q)\ePK{}q^{f(z)}\p(z;q) + q^{f(z)}\c(z;q) \]
in $\PzZKh$ and thus the result follows by translation of the integrand and Lemma \ref{Bintegralcongruence},
\[
\int_{K+B}^{a+K+B}q^{f(z)}\p(z;q)dz
=\int_B^{a+B}q^{f(z+K)}\p(z+ K;q)dz\ePK\int_B^{a+B}q^{f(z)}\p(z;q)dz + \int_B^{a+B}q^{f(z)}\c(z;q)dz
\]
\end{proof}

\begin{thm}\label{generalcarry-over}  Suppose that $\p(z;q),\c(z;q)\in\PzZKh$ satisfy $\p(z+K;q)\ePK\p(z;q)+\c(z;q)$ and $\c(z+K;q)\ePK\c(z;q)$. If $f(z)\in\ZK[z]$ and  $a,\alpha\in\ZK$, then in $\ZKh$
\[
\int_{a+B}^{a+\alpha{}K+B}q^{f(z)}\p(z;q)dz
\ePK\alpha\int_B^{K+B}q^{f(z)} \p(z;q)dz + \alpha\int_B^{a+B}q^{f(z)}\c(z;q)dz
    + \frac12\alpha(\alpha-1)\int_B^{K+B}q^{f(z)}\c(q;z)dz
\]
\end{thm}
\begin{proof}
Since $q^{f(z)}\p(z;q)\in\PzZKh$, thus the integral on the left hand side of the equality to be proved is an element of $\ZKh$, all of whose coefficients depend polynomially  on $\alpha$, that is, it is an element of $\PalphaZKh$. The right hand side is also an element of $\PalphaZKh$ and so by Lemma \ref{pzqmodPhicharacterisation}, it is sufficient to prove their congruence modulo $\Phi_K$ for all $\alpha\in\Z$ in order to guarantee their congruence modulo $\Phi_K$ for all $\alpha\in\ZK$.

\noindent By Theorem \ref{carry-overlemma2} applied to $\c(z;q)$, $\int_{a+B}^{a+K+B} q^{f(z)} \c(z;q) dz \ePK\int_B^{K+B} q^{f(z)} \c(z;q)dz$. Combining several intervals of length $K$, thus for all $m\in\Z$,
\[\int_{a+B}^{a+mK+B} q^{f(z)} \c(z;q) dz \ePK{}m\int_B^{K+B} q^{f(z)} \c(z;q)dz\]
Consequently, applying Theorem \ref{carry-overlemma2} to $\p(z;q)$,
\begin{align*}\int_{a+mK+B}^{a+(m+1)K+B}q^{f(z)}\p(z;q)dz
&\ePK\int_{B}^{K+B}q^{f(z)}\p(z;q)dz+\int_B^{a+mK+B}q^{f(z )}\c(z;q) dz\\
&\ePK\int_{B}^{K+B}q^{f(z)}\p(z;q)dz+\int_B^{a+B}q^{f(z )}\c(z;q) dz+m\int_B^{K+B}q^{f(z )}\c(z;q) dz
\end{align*}
Summing this over $m$ ranging from $0,1,\ldots,n-1$ gives the result in the theorem for $\alpha=n\in\N$, while summing it for $m$ ranging over $-1,\ldots,-n$ gives the result for $\alpha=-n$, $n\in\N$.
\end{proof}

\noindent When $\p(z;q)=1$, $\c(z;q)=0$ we obtain Lemma 5.0.2 \cite{Law95} as a corollary of Theorem \ref{generalcarry-over}.
\begin{lem}\label{1dcarry-overcorollary} If $f(z)\in\ZK[z]$ and $a\eK{}b$ in $\ZK$, then $\int\limits_{a+ B }^{b+ B } q^{f(z)} dz \ePK \frac{b-a}{K} \int\limits_B ^{K+ B} q^{f(z)} dz$ in $\ZKh$.
\end{lem}

\noindent Similarly this holds with an additional parameter.

\begin{lem}\label{1dcarry-overcorollaryzw} If $f(z,w)\in\ZK[z,w]$ and $a(z)\eK{}b(z)$ in $\PzZK$, then in $\PzZKh$,
\[\int\limits_{a(z)+ B }^{b(z)+ B } q^{f(z,w)} dw \ePK \frac{b(z)-a(z)}{K} \int\limits_B ^{K+ B} q^{f(z,w)} dw\]
\end{lem}

\subsection{\textsc{Bernoulli Multiplication Theorems}}

\begin{prop}\label{complementaryBintegrals}  For an even polynomial $p(z) \in \R[z^2]$ then $\int_{B}^{z+ B} p(w) dw =\B_z[ \int_0^z p(w) dw]+\frac12p(0)$. Furthermore
$\int_{B}^{1- z+ B} p(w) dw =p(0) -\int_{B}^{z+ B} p(w) dw$ and in particular $\int_{B}^{1+ B} p(w) dw =p(0)$ .
Similarly for $\p(z;q)\in\Rztwoh$ then $\int_{B}^{1-z+ B}\p(z;q)dz = \p(0;q)-\int_B^{z+B}\p(z;q)dz$.
\end{prop}
\begin{proof} Since $p(z)$ is even, thus  $p(z)= \sum_{n=0}^{\frac{1}{2}\deg(p)} a_{n}z ^{2n} $ for some $a_n\in\R$. Integrating,
\begin{align*}
 \int_{B}^{z+ B} &p(w) dw =\int_0^z \mathbb{B}_w [p(w)] dw = \int_0^z  \left(\sum_{n=0}^{\frac{1}{2}\deg(p)} a_{n} B_{2n}(w)
 \right) dw \\
 &=  \sum_{n=0}^{\frac{1}{2} \deg(p)} \frac{a_{n}}{2n+1} (B_{2n+1} (z) - B_{2n+1})
 =\sum_{n=0}^{\frac{1}{2} \deg(p)} \frac{a_{n}}{2n+1} B_{2n+1} (z) - a_0B_1
 =  \mathbb{B}_z \left[\int_0^z p(w) dw \right] +\frac12p(0)
\end{align*}
where the last but one step follows from the vanishing of odd-indexed Bernoulli numbers except for $B_1=-\frac12$ while $a_0 = p(0)$. By Lemma \ref{Bpolyprop}(3), $B_{2n+1}(1-z)=-B_{2n+1}(z)$ from which it follows that  $\int_{B}^{1- z+ B} p(w) dw =p(0) -\int_{B}^{z+ B} p(w) dw$. Setting $z=0$ gives $\int_B^{1+B}p(w)dw=p(0)$.
    \end{proof}

\begin{lem}\label{Bernoullimultiplicationlemma} For any $m\in \Z^{\geq0}$ and $x\in\R$
\begin{enumerate}
    \item Multiplication Theorem: $\sum\limits_{r=0}^{m-1}  B_n \left( x +  \frac{r}{m} \right) = m^{1-n} B_n(mx)$;
    \item Reverse Multiplication Theorem: $\sum\limits_{r=0}^{m-1} B_n \left( x - \frac{r}{m}\right) = (-1)^n m^{1-n} B_n \big(m (1-x)\big) $.
\end{enumerate}
\end{lem}
\begin{proof}  By the generating function for Bernoulli polynomials \eqref{Bpolygenfn},
\begin{align*}
    \sum_{r=0}^{m-1} \sum_{n=0}^\infty B_n \left(x + \frac{r}{m} \right) \frac{t^n}{n!}
    &=\sum_{r=0}^{m-1} \frac{te^{t(x+ \frac{r}{m}  )}}{e^t -1}
    = \frac{te^{tx}}{e^t -1} \cdot \frac{ e^t-1  }{ e^{\frac{t}{m}}-1}
    =m \frac{\frac{t}{m} e^{tx }  }{e^{t/m }-1} = m \sum_{n=0}^\infty  B_n (mx)\frac{(\frac{t}{m})^n }{n!}\>,\\
 \sum_{r=0}^{m-1}\sum_{n=0}^\infty{}B_n \left( x - \frac{r}{m} \right)   \frac{t^n }{n!}
 &=  \sum_{r=0}^{m- 1}\frac{t e^{ t (x - \frac{r}{m}) } }{e^t -1}
 = \frac{t e^{t x} }{e^t -1}\cdot \frac{1-e^{-t}}{1 - e^{-\frac{t}{m}}}
 =m\frac{-\frac{t}{m} e^{ t(x-1) }}{e^{-\frac{t}{m}}-1}
 = m \sum_{n=0}^\infty B_n(m(1-x)) \frac{(-\frac{t}{m})^n }{n!}
\end{align*}
from which the results follow by equating coefficients of $t^n$.
\end{proof}

\noindent This lemma can be interpreted by saying that while the Bernoulli polynomials are not homogeneous of degree $n$, averaging over cosets $\frac{r}{m}$ exactly corrects the failure,
\[
B_n(mx)=m^n\cdot\frac1m\sum_{r=0}^{m-1}B_n\left(x+\frac{r}{m}\right)
\]

\begin{thm}\label{Bernoullimultiplicationlemma2} Let $m \in \N$ and $z_0 \in \R$. Then,
\begin{align*}
    &\textbf{\textit{1. }}\sum_{r=0}^{m-1} \left(\int_{ \frac{r}{m} + B  }^{z_0 + \frac{r}{m} + B }   f(w)dw \right) = \int_B ^{mz_0 + B } f(m^{-1}w  )dw \>,\quad\hbox{ for any $f(z) \in \R[z]\>;$}\\
    &\textbf{\textit{2. }} \sum_{r=0}^{m-1} \left( \int_{-\frac{r}{m} + B }^{z_0 - \frac{r}{m} + B }f(w)dw \right) = \int_{m (1-z_0)+B}^{m+ B } f(m^{-1} w) dw\>, \quad\hbox{ for any \textbf{even }$f(z) \in \R[z^2]$.}
\end{align*}
\end{thm}
\begin{proof} Both properties are linear in $f(z)$ and so it suffices to prove them for $f(z)=(n+1)z^n$, $n\in\Z^{\geq0}$, where in the second property, $n$ is assumed even. In this case, the properties read
\begin{align*}
\sum_{r=0}^{m-1}\left[B_{n+1}\left(z_0+\frac{r}{m}\right)-B_{n+1}\left(\frac{r}{m}\right)\right]
&=m^{-n}\big(B_{n+1}(mz_0)-B_{n+1}(0)\big)\quad\forall{}n\in\Z^{\geq0}\\
\sum_{r=0}^{m-1}\left[B_{n+1}\left(z_0-\frac{r}{m}\right)-B_{n+1}\left(-\frac{r}{m}\right)\right]
&=m^{-n}\big(B_{n+1}(m)-B_{n+1}(m-mz_0)\big)\quad\forall\hbox{ even $n\in\Z^{\geq0}$}
\end{align*}
which follow from Lemma \ref{Bernoullimultiplicationlemma} by subtracting applications at $x=0,z_0$.
\end{proof}

\noindent Note that similar statements hold for power series with polynomial coefficients, the first for $\p(z;q)\in\Rzh$ and the second for $\p(z;q)\in\Rztwoh$, for example in the case $\mathfrak{p}(z;q) = q^{ \chi z^2}$.

\begin{rem} The summation analogue of the first fact is the identity
\[
    \sum_{z=0}^{m z_0 -1} f( m ^{-1}z) = \sum_{r=0}^{m-1} \sum_{z=0}^{z_0 -1} f \big(m^{-1} ( m z + r  ) \big)= \sum_{r=0}^{m -1} \sum_{z=0}^{z_0 -1} f \left(z+ \frac{r}{m} \right)
\]
which holds for any function $f(z):\R\rightarrow\R$.
\end{rem}

\begin{thm}\label{Bernoullimultiplicationlemma3} Suppose that $m\in\N$, $z\in\R$ and $f(z)\in\R[z^2]$ is an even polynomial. Then
\begin{align*}
    &\textbf{\textit{1. }}\sum_{r=0}^{m-1} \int_B ^{z + \frac{r}{m }+ B } f(mw) dw = \int_B ^{mz + B } f(w) dw + \frac{m-1}{2}f(0),\\
    &\textbf{\textit{2. }}\sum_{r=0}^{m-1} \int_B ^{z- \frac{r}{m} + B } f(mw) dw = \int_{ m(1-z)+ B  }^B f(w) dw + \frac{m+1}{2}f(0).
\end{align*}
Analogous statements hold for power series $\p(z;q) \in \Rztwoh$ with \textbf{even} polynomial coefficients.
\end{thm}
\begin{proof} Both properties are linear in $f(z)$ and so it suffices to prove them for $f(z)=(2n+1)z^{2n}$, $n\in\Z^{\geq0}$, where in the second property, $n$ is assumed even. In this case, $f(0)=\delta_{n=0}$ and the properties read
\begin{align*}
m^{2n}\sum_{r=0}^{m-1} \left[B_{2n+1}\left(z+\frac{r}{m }\right)-B_{2n+1}\right] &= B_{2n+1}(mz)-B_{2n+1}+ \frac{m-1}{2}\delta_{n=0},\\
m^{2n}\sum_{r=0}^{m-1} \left[B_{2n+1}\left(z-\frac{r}{m}\right)-B_{2n+1}\right] &= B_{2n+1}-B_{2n+1}\big(m(1-z)\big) + \frac{m+1}{2}\delta_{n=0}.
\end{align*}
For $n>0$, $B_{2n+1}=0$ and the properties follow immediately from Lemma \ref{Bernoullimultiplicationlemma} with $2n+1$ in place of $n$. When $n=0$, $B_1(x)=x-\frac12$ and so the equalities read $\sum_{r=0}^{m-1}(z+\frac{r}{m})=mz+\frac{m-1}2$ and $\sum_{r=0}^{m-1}(z-\frac{r}{m})=-m(1-z)+\frac{m+1}2$ which are correct. Since symbolic integration of power series is defined term-by-term, the corresponding identities hold on any power series with even polynomial coefficients.
\end{proof}

\section{\textsc{One-dimensional result}}

In this section, we prove (1)-(4) of Conjecture \ref{quantumpickconjecture} for $d=1$. For $x\in\R$, denote its fractional part by $\{x\}=x-\lfloor{}x\rfloor$.

\begin{defn} \label{defBint} Let $\alpha,\beta\in\R$. For $f(z)\in \R[z]$,  define the symbolic $B$-integral over an interval by
\begin{align*}
    &\Bint_{[\alpha,\beta]} f(z) dz := \int_{ \{-\alpha\} + B }^{1- \{\beta\}+ B  } f(z) dz\in\R\>.
\end{align*}
For $\p(z;q)\in\Rzh$, $\Bintt_{[\alpha,\beta]} \p(z;q) dz\in\Rh$ is defined by term-by-term integration of the power series.
   \end{defn}
\noindent If $\alpha,\beta\in\Q$ and $f(z)\in\Q[z]$ then $\Bintt_{[\alpha,\beta]}f(z)dz\in\Q$. By Lemma \ref{Bintegralbelonging}, if $\alpha,\beta\in\ZK$ and $f(z)\in\PzZK$, then $\Bintt_{[\alpha,\beta]}f(z)dz\in\ZK$.
Similarly if $\alpha,\beta\in\ZK$ and $\p(z;q)\in\PzZKh$ then $\Bintt_{[\alpha,\beta]}\p(z;q)dz\in\ZKh$.

\begin{lem}\label{BIntegralProperties} Suppose $f(z)\in\Q[z]$.
\begin{enumerate}
    \item Localization on a Point: $\Bintt_{ [\alpha,\alpha] } f(z)dz =f(0)\delta_{\alpha \in \Z}$.
    \item Additivity up to Double Counting:
    $\Bintt_{[\alpha,\beta]}\p(z;q)dz+\Bintt_{[\beta,\gamma]}f(z)dz=\Bintt_{[\alpha,\gamma]} f(z) dz + f(0)\delta_{\beta \in \Z}$.
    \item Central inversion: $\Bintt_{[-\beta,-\alpha]}f(-z)dz=\Bintt_{[\alpha,\beta]}f(z)dz$.
    \item If $\{\alpha\} + \{\beta\}=1$ and $f$ is even, then
    $\Bintt_{[0,\alpha]}f(z)dz+\Bintt_{[0,\beta]}f(z)dz =f(0)$.
    \item If $f$ is even, $\Bintt_{ [0,\alpha]} f(z) dz +\Bintt_{ [0,-\alpha] } f(z) dz =(1 + \delta_{\alpha \in \Z })f(0)$.
    \item Almost Independence: $\Bintt_{[K\alpha,K\beta]} q^{f(z)} dz$ defines a squak almost independent of $K$.
\end{enumerate}
\end{lem}
\begin{proof}
\textbf{\textit{1.}} By definition, $\Bintt_{ [\alpha,\alpha] } f(z)dz=\int_{\{-\alpha\}+B}^{1-\{\alpha\}+B}\!f(z)dz$.
When $\alpha\notin\Z$, $\{-\alpha\}=1-\{\alpha\}$ and so this integral is immediately zero. When $\alpha\in\Z$, $\{\alpha\}=\{-\alpha\}=0$ and so the integral becomes $\int_B^{1+B}f(z)dz=f(0)$.\\
\textbf{\textit{2.}} By definition the LHS is
$\int_{\{-\alpha\}+B}^{1-\{\beta\}+B}f(z)dz+\int_{\{-\beta\}+B}^{1-\{\gamma\}+B}f(z)dz
=\int_{\{-\alpha\}+B}^{1-\{\gamma\}+B}f(z)dz+\int_{\{-\beta\}+B}^{1-\{\beta\}+B}f(z)dz$ by additivity of symbolic integrals under splitting their interval and the result follows from part 1. \\
\textbf{\textit{3.}} By definition, $\Bintt_{[\-\beta,-\alpha]}f(-z)dz=\int_{ \{\beta\} + B }^{1- \{-\alpha\}+ B  } f(-z) dz$. The result follows by the central inversion of symbolic integrals in Lemma \ref{symbolicinttranslationinversion}.
\\
\textbf{\textit{4.}} The left hand side becomes $\int_{B}^{\{\beta\}+B}f(z)dz+\int_{B}^{\{\alpha\}+B}f(z)dz=f(0)$ by Lemma \ref{complementaryBintegrals}.\\
\textbf{\textit{5.}} The LHS is $\int_{B}^{1-\{\alpha\}+B}f(z)dz+\int_{B}^{1-\{-\alpha\}+B}f(z)dz
= \int_{B}^{1-\{\alpha\}+B}f(z)dz+\int_{B}^{\{\alpha\}+B}\!f(z)dz+\int_{\{\alpha\}+B}^{1-\{-\alpha\}+B}f(z)dz$. The sum of the first two terms is $f(0)$ by Lemma \ref{complementaryBintegrals} while the last term evaluates using part 1.\\
\textbf{\textit{6.}} By definition, the dependence on $K$ of the symbolic $B$-integral is only through the limits $\{-K\alpha\}$ and $\{K\beta\}$, that is, it depends only on $K$ modulo $\cN$ where $\cN$ is the lcm of the denominators of $\alpha$,$\beta$.
\end{proof}

\begin{thm}\label{1dqpt} (\textbf{\textit{Conjecture \ref{quantumpickconjecture} in 1D.}}) Let $\alpha,\beta \in \Q$ with $\alpha<\beta$ and $f(z) \in \Q[z]$ be a quadratic. Then
\[S_K([\alpha,\beta])\equiv\sum\limits_{z\in[K\alpha,K\beta]\cap\Z}q^{f(z)}\]
defines a squak with level one $\G$-decomposition
\[
    S_K([\alpha,\beta]) \ePK \Bint_{[K\alpha,K\beta]} q^{f(z)} dz
    +(\beta-\alpha) \left( \frac{(a)_K}{K}  \right) q^{f_0} \G =Z^{(0)}_K([\alpha,\beta])+Z^{(1)}_K([\alpha,\beta])\G
\]
for all but finitely many primes $K$, where $f_0$ is the value of $f(z)$ at its critical point and $a$ is the leading coefficient in $f(z)$.
\end{thm}
\begin{proof} For any odd prime $K$ not dividing the denominators of the coefficients of $f(z)$, then $f(z)\in\ZK[z]$ and so $S_K([\alpha,\beta])\in\ZKh$ by Lemma \ref{qpoweralpha}. Hence $S_K([\alpha,\beta])$ is a squak. Writing the sum as a symbolic integral by Lemma \ref{EulerMaclaurin} and then splitting the interval of integration,
\begin{align*}
    S_K ([\alpha,\beta])
    &= \sum_{z= \lceil K\alpha \rceil  }^{\lfloor K\beta \rfloor } q^{f(z)}
    = \int_{\lceil K\alpha \rceil + B } ^{ \lfloor K\beta \rfloor + 1 + B  } q^{f(z)} dz\\
    &= \int_{\lceil K\alpha \rceil + B } ^{\lceil K\alpha \rceil -K\alpha + B  } q^{f(z)} dz + \int_{\lceil K\alpha \rceil- K\alpha + B } ^{ \lfloor K\beta \rfloor + 1- K\beta + B  } q^{f(z)} dz + \int_{\lfloor K\beta \rfloor -K\beta +  1  + B } ^{ \lfloor K\beta \rfloor + 1 + B  } q^{f(z)} dz
\end{align*}
Applying Lemma \ref{1dcarry-overcorollary} to the first and last terms and noting that $\{x\}=x-\lfloor{}x\rfloor$ while $\{-x\}=\lceil{}x\rceil-x$ gives
\begin{align*}
  S_K ([\alpha,\beta])&\ePK  \int_{ \{-K\alpha\}+ B  }^{ 1- \{K\beta\}+ B  } q^{f(z)} dz
  +(\beta - \alpha) \int_B ^{K+ B } q^{f(z )}dz\\
  &\ePK \Bint_{[K\alpha,K\beta]} q^{f(z)} dz
    +(\beta-\alpha) \left( \frac{(a)_K}{K}  \right) q^{f_0} \G
\end{align*}
by Definition \ref{defBint} and Lemma \ref{BinteasyGauss}, for $K$ so that $a\not\equiv_K0$. That this is a level one $\G$-decomposition follows from Lemma \ref{BIntegralProperties}(6) and Lemma \ref{Legendresymbolalmostindep}.
\end{proof}

\noindent Theorem \ref{1dqpt} provides a level one $\G$-decomposition of the discrete sum squak
\[S_K([\alpha,\beta]) \ePK Z^{(0)}_K([\alpha,\beta])+Z^{(1)}_K([\alpha,\beta])\G\]
where \[Z_K^{(0)}([\alpha,\beta])=\Bint_{[K\alpha,K\beta]} q^{f(z)} dz\] is the codimension zero contribution, a squak almost independent of $K$ (it is a $B$-integral over $K$ times the interval) while
\[Z^{(1)}_K([\alpha,\beta])=(\beta-\alpha) \left( \frac{(f_1)_K}{K}  \right) q^{f_0}\]
is the codimension one (vertex) contribution, a squak almost independent of $K$ (a multiple of the length of the interval).
Under concatenation of intervals,  $Z_K^{(1)}$ is additive (since length is additive) while by Lemma \ref{BIntegralProperties}(2),
\[
Z^{(0)}_K([\alpha,\beta])+Z^{(0)}_K([\beta,\gamma])=Z^{(0)}_K([\alpha,\gamma])+q^{f(0)}\delta_{K\beta\in\Z}
\]
On the other hand, $S_K$ has an error from being additive which is the possible contribution of the overlapping interior point,
\[
S_K([\alpha,\beta])+S_K([\beta,\gamma])=S_K([\alpha,\gamma])+q^{f(K\beta)}\delta_{K\beta\in\Z}
\ePK S_K([\alpha,\gamma])+q^{f(0)}\delta_{K\beta\in\Z}
\]
since $f(K\beta)\eK{}f(0)$ and so $q^{f(K\beta)}\ePK{}q^{f(0)}$ by Lemma \ref{modKmodphiK}.

\section{\textsc{Two-dimensional result}}

A sum over a rectangular domain of a separable summand $q^{ f_1 z^2 + f_3 w^2  }$ splits immediately as a product of one-dimensional sums,
\begin{align*}
        S_K ( [\alpha,\beta]\times [\gamma,\delta ] )
        &\equiv_{\Phi_K}S^z_K ( [\alpha,\beta] )\cdot S^w_K ([\gamma, \delta])\\
        &\equiv_{\Phi_K} \underbrace{Z^{(0)z}_K( [\alpha,\beta] )Z^{(0)w}_K ([\gamma,\delta])    }_{Z_K  (\rectangle)}+
        (\text{cross terms}) \G+   \underbrace{(\beta-\alpha)(\delta-\gamma)}_{A (\rectangle )} \left(\frac{ \overline{f_1} }{K} \right) \left(\frac{ \overline{f_3} }{K} \right) \G^2
    \end{align*}
implying that for any quadratic $f(z,w)$ with no cross term, one has a level two $\G$ decomposition of the sum on a rectangular domain whose face term is
$Z^{(0)}_K ([ \alpha,\beta ]\times [\gamma,\delta])=Z^{(0)z}_K ( [\alpha,\beta] ) Z^{(0)w}_K ( [\gamma,\delta] )=
 \left( \Bintt_{  [\alpha, \beta ] }  q^{f_1 z^2 }dz \right) \left( \Bintt_{ [\gamma, \delta ] }  q^{f_3 w^2 }dw  \right)$ which one could decide to denote by $ \Biint _{ [\alpha,\beta]\times[\gamma, \delta]} q^{ f_1 z^2 + f_3 w^2  } dw dz$.
However this method doesn't work for domains of other shapes or non-separable summands.

\subsection{\textsc{Study of non-homogeneous two-variable quadratics}}
Throughout this section, $f$ will denote a quadratic function of two variables with rational coefficients, $f(z,w) := f_1 z^2 + f_2 zw + f_3 w^2 + f_4 z + f_5 w + f_6$. We assume that the homogeneous part of $f$ defines a non-degenerate quadratic form $f_h (z,w) := f_1 z^2 + f_2 zw + f_3 w^2$. Define two numerical quantities
\begin{align*}
    &\Delta:= 4f_1 f_3-f_2^2\>  ,   &&\bfD := 4f_1 f_3 f_6 - f_3 f_4 ^2 -f_2^2 f_6 - f_1 f_5^2  + f_2 f_4 f_5\> .
\end{align*}
$\Delta\not=0$ is the determinant of the Hessian matrix of $f$ (it is equal to minus the discriminant of the one-variable quadratic $f_h(z,1)$) while $\bfD$ is (half) the determinant of the Hessian matrix of the homogeneous quadratic form in three variables corresponding to $f$ namely $t^2f(\frac{z}{t},\frac{w}{t})$. Both are invariant under interchange of $z$ and $w$. Notice that $\bfD=0$ if $f=f_h$.

\noindent Let $\bfA$ denote the matrix of the quadratic form $f_h$, that is
$\bfA=\begin{pmatrix}
        f_1 & \frac12f_2 \\
        \frac12f_2 & f_3
    \end{pmatrix}$
and put $\bfz=\begin{pmatrix}z\\ w\end{pmatrix}$, $\bfb=\begin{pmatrix}f_4\\ f_5\end{pmatrix}$.
Then    $f_h(\bfz)=\bfz^T\bfA\bfz$ while $f(\bfz)=\bfz^T\bfA\bfz+\bfz^T\bfb+f_6$. Under a translation,
\[
f(\bfz+\bfz_0)=\bfz^T\bfA\bfz+\bfz^T(2\bfA\bfz_0+\bfb)+(\bfz^T_0\bfA\bfz_0+\bfz^T_0\bfb+f_6)
\]
Choosing $\bfz_0$ so as to make the linear term vanish, $\bfz_0=-\frac12\bfA^{-1}\bfb$ then $f(\bfz+\bfz_0)=f_h(\bfz)+\frac{\bfD}{\Delta}$. Explicitly, $\bfz_0=(z_0,w_0)^T$ with
\[f(z,w) = f_h (z - z_0, w - w_0) + \frac{\bfD}{\Delta}
\quad\hbox{ where }z_0=\frac{f_2f_5-2f_3f_4}{\Delta}\>,\quad{}w_0=\frac{f_2f_4-2f_1f_5}{\Delta}
\]
Here $(z,w)=(z_0,w_0)$ is the unique stationary point of $f(z,w)$ and the value attained there is $\frac{\bfD}{\Delta}$.

\begin{lem} \label{generalquadraticslice} For arbitrary $\bfa,\bfm\in\R^2$ with $f_h(\bfm)\not=0$,
\[
    f(t\bfm+\bfa ) = f_h (\bfm) \left( t+\frac{\bfm\cdot(\bfA\bfa+\frac12\bfb)}{f_h(\bfm)}\right)^2
    +\frac{\Delta}{4f_h(\bfm)}\big(\bfm^\perp\cdot(\bfa-\bfz_0)\big)^2+\frac{D}{\Delta}\>,
\]
where $\bfm^\perp$ denotes the rotation of $\bfm$ through $\frac\pi2$.
\end{lem}
\begin{proof} Since $f$ is a translation of $f_h$ up to a constant and $f_h(\bfx)=\bfx^T\bfA\bfx$,
\begin{align*}
    f(t\bfm+\bfa)&=f_h(t\bfm+\bfa-\bfz_0)+\frac{D}{\Delta}
    =(t\bfm^T+\bfa^T-\bfz^T_0)\bfA(t\bfm+\bfa-\bfz_0)+\frac{D}{\Delta}\\
    &=t^2f_h(\bfm)+2t\bfm^T\bfA(\bfa-\bfz_0)+f_h(\bfa-\bfz_0)+\frac{D}{\Delta}\\
    &=f_h(\bfm)\left(t+\frac{\bfm^T\bfA(\bfa-\bfz_0)}{f_h(\bfm)}\right)^2
    +\frac{\Delta}{4f_h(\bfm)}\big(\bfm^\perp\cdot(\bfa-\bfz_0)\big)^2+\frac{D}{\Delta}
\end{align*}
where in the last step we use the identity
$(\bfm^T\bfA\bfm)(\bfx^T\bfA\bfx)\kern-.5pt-\kern-.5pt(\bfm^T\bfA\bfx)^2=\frac{\Delta}{4}(m_2x_1-m_1x_2)^2=\frac{\Delta}{4}(\bfm^\perp\cdot\bfx)^2$.
\end{proof}

\begin{lem}\label{Gaussslice} For odd $K$ prime with $f(z,w)\in\ZK[z,w]$, $m_1,m_2,a_1,a_2\in\ZK$ and $\Delta,f_h(m_1,m_2)\neK0$,
\[
\int_B^{K+B}q^{f(m_1x+a_1,m_2x+a_2)}dx\ePK \left( \frac{(f_h (m_1,m_2))_K}{K} \right)\G
q^{\frac{D}{\Delta}+\frac{\Delta(m_1(a_2-w_0)+m_2(z_0-a_1))^2}{4f_h(m_1,m_2)}}
\]
\end{lem}
\begin{proof}
    The proof follows immediately by combining Lemma \ref{generalquadraticslice} with Lemma \ref{easyGauss}.
\end{proof}

\begin{lem}\label{generalquadraticsimplified}  Let $K$ be an odd prime such that $f(z,w)\in\ZK[z,w]$ and $f_1,f_3,\Delta\neK0$. Then
    \[
    \int_B^{K+B} q^{f(z,w)}dz\ePK\left(\frac{(f_1)_K}{K} \right)\G{}q^{\frac{D}{\Delta}+\frac{\Delta(w-w_0)^2}{4f_1}}\>,\quad
    \int_B^{K+B} q^{f(z,w)}dw\ePK\left(\frac{(f_3)_K}{K} \right)\G{}q^{\frac{D}{\Delta}+\frac{\Delta(z-z_0)^2}{4f_3}}\>.
    \]
while $\int_B^{K+B}\int_B^{K+B}q^{f(z,w)}dw dz\ePK\left(\frac{(\Delta)_K}{K}\right)\G^2q^{\frac{\bfD}{\Delta}}$.\end{lem}
\begin{proof}
Applying Lemma \ref{Gaussslice} with $\bfm=(1,0),(0,1)$ gives the two single integrals. For the iterated integral,
\[
\int_B^{K+B}\!\!\!\!\int_B^{K+B}\!\!\!q^{f(z,w)}dw dz
\ePK\left(\frac{(f_3)_K}{K} \right)\!\G{}q^\frac{D}{\Delta}\int_B^{K+B}q^\frac{\Delta(z-z_0)^2}{4f_3}dz
=\left(\frac{(\Delta)_K}{K} \right)\G^2q^{\frac{\bfD}{\Delta}}
\]
performing a second symbolic integration and using Lemma \ref{easyGauss} again,
\end{proof}

\subsection{\textsc{Line Segments of arbitrary rational slope}}
In this subsection, we consider sums over those lattice points on the scaling of a line segment, rather than the interior of a polygon. Let $\cL$ denote the line segment joining the rational points $(\alpha,\gamma)$ and $(\beta,\delta)$.  Let $\cLb$ denote the extended line, to infinity in both directions. Write $\bfm=\begin{pmatrix}m_1\cr m_2\end{pmatrix}$, a reduced vector in $\Z^2$ parallel (and not anti-parallel) to $\begin{pmatrix}\beta-\alpha\cr \delta-\gamma\end{pmatrix}$ and thus to $\cL$. Let $\bfm^\perp=\begin{pmatrix}-m_2\cr m_1\end{pmatrix}$ be the rotation of $\bfm$ through $\frac\pi2$; it will be a reduced vector in $\Z^2$ normal to $\cL$.

\noindent Then the extended line $\cLb$ has equation $\bfr\cdot\bfm^\perp=c$ where $c=m_1\gamma-m_2\alpha=m_1\delta-m_2\beta$. By definition, $(\beta-\alpha):(\delta-\gamma)=m_1:m_2$. Let $\lambda(\cL)\in\Q^{\geq 0}$ be the {\it length parameter} defined so that $\beta-\alpha=m_1\lambda(\cL)$ and $\delta-\gamma=m_2\lambda(\cL)$. In the case of a horizontal or vertical segment, $(m_1,m_2)=(\pm1,0)$ or $(0,\pm1)$ respectively, and the length parameter $\lambda(\cL)$ will be the actual length of the segment $\cL$.

\begin{thm}\label{lineinvariant} Suppose that $f(z,w)\in\Q[z,w]$ with $f_h(m_1,m_2)\not=0$. Then the sum
\[S_K(\cL)\equiv\sum\limits_{(z,w)\in{}K\cL\cap\Z^2}q^{f(z,w)}\]
defines a squak which is trivial for $c\not\in\Z$ and which for $c\in\Z$ has level one $\G$-decomposition,
\begin{align*}\label{diagonalzperspective}
    S_K (\cL)\ePK \Bint_{[Ka,Kb]} q^{   f(m_1t, m_2 t)} dt
    +\lambda(\cL) \left( \frac{(f_h (m_1, m_2))_K}{K} \right)  q^{\frac{\textbf{D}}{\Delta} + \frac{\Delta(m_1w_0-m_2z_0)^2}{4f_h (m_1,m_2)} } \G \equiv Z^{(0)}_K(\cL)+Z^{(1)}_K(\cL)\G
\end{align*}
for all but finitely many $K$, where $a,b\in\Q$ are determined by the endpoints of $\cL$.
\end{thm}
\begin{proof}
 When $K$ prime is such that $f(z,w)\in\ZK[z,w]$ then $f(z,w)\in\ZK$ at all integer points $(z,w)$, so that by Lemma \ref{qpoweralpha}, the summand $q^{f(z,w)}$ lies in $\ZKh$ and hence its sum $S_K(\cL)$ defines a squak.

 \noindent Since $m_1,m_2$ are coprime, we can choose $n_1,n_2\in\Z$ so that $m_1n_1+m_2n_2=1$; in the case of a horizontal or vertical line segment, $\bfm=(\pm1,0),(0,\pm1)$, we choose $n_1=m_1$, $n_2=m_2$.   Note that
 the extended line $K\cLb:\bfr\cdot\bfm^\perp=cK$ will have a non-empty intersection with $\Z^2$ precisely when $cK\in\Z$.
 When $c\in\Q\backslash\Z$, then for all sufficiently large prime $K$, $cK\not\in\Z$ and so $K\cL\cap\Z^2=\emptyset$ making $S_K(\cL)$ a trivial squak. In the case that $cK\in\Z$, then $(-cn_2K,cn_1K)\in{}K\cLb\cap\Z^2$ and so
 \[K\cLb\cap\Z^2=\begin{pmatrix}-cn_2K\cr cn_1K\end{pmatrix}+\Z\begin{pmatrix}m_1\cr m_2\end{pmatrix}\>.\]
This gives a parametrisation of $K\cLb$ by $z=m_1x-cn_2K, w=m_2x+cn_1K$ by the parameter $x\in\R$ where integer points $(z,w)\in\Z^2$ correspond precisely to $x\in\Z$. The section $K\cL$ of $K\cLb$ is given by $Ka\leq{}x\leq{}Kb$ where
\[
a=\frac{\alpha+cn_2}{m_1}=\frac{\gamma-cn_1}{m_2},\qquad{}b=\frac{\beta+cn_2}{m_1}=\frac{\delta-cn_1}{m_2}
\]
Note that by the choice of signs on $m_1,m_2$, always $a\leq{}b$. In the case of a horizontal line segment, $a=m_1\alpha$, $b=m_1\beta$ while for a vertical line segment, $a=m_2\gamma$, $b=m_2\delta$. Hence the line segment sum can be written as
\[
S_K(\cL)=\sum_{x\in{}K[a,b]\cap\Z}q^{f(m_1x-cn_2K,m_2x+cn_1K)}
\]
When $f(z,w)\in\ZK[z,w]$, the exponent in the summand is congruent modulo $K$ to $f(m_1x,m_2x)$ in $\ZK$; hence by Lemma \ref{modKmodphiK}
\[
S_K(\cL)\ePK\sum_{x\in{}K[a,b]\cap\Z}q^{f(m_1x,m_2x)}
\]
The exponent in this sum is $f(m_1x,m_2x)\in\Q[x]$ which is quadratic and by Lemma \ref{generalquadraticslice} with $\bfa={\bf0}$, it has leading coefficient $f_h(m_1,m_2)$ and the value at the critical point is
$f_0=\frac{\Delta}{4f_h(\bfm)}\big(\bfm^\perp\cdot\bfz_0\big)^2+\frac{D}{\Delta}$.
Applying Lemma \ref{1dqpt}, we obtain the level one $\G$ decomposition
\[
S_K(\cL)\ePK\Bint_{[Ka,Kb]} q^{f(m_1z,m_2z)} dz
    +(b-a) \left( \frac{(f_h(m_1,m_2))_K}{K}  \right) q^{f_0} \G
\]
which gives the result stated, upon noting that $b-a=\lambda(\cL)$.
\end{proof}

\noindent {\bf Remark 1.} Note that interchanging the endpoints $(\alpha,\gamma)$ and $(\beta,\delta)$ changes $\bfm$ to $-\bfm$. We can choose $\bfn$ to be similarly reversed in sign. Then $c$ changes sign but $\lambda$ stays the same while $a,b$ become $-b,-a$ respectively. Certainly $S_K(\cL)$ stays the same, while, being homogeneous, $f_h(-m_1,-m_2)=f_h(m_1,m_2)$. Finally, the symbolic integral giving $Z_K^{(0)}(\cL)$ changes to a central inversion of itself, which by Lemma \ref{BIntegralProperties}(3) has unchanged value.

\medskip\noindent{\bf Remark 2.} In the case that $f_h(m_1,m_2)=0$ but $f$ is not constant in the direction $\bfm$, that is $f_4m_1+f_5m_2\not=0$, then a level one $\G$-decomposition for $S_K(\cL)$ still exists, simply $Z^{(1)}_K(\cL)=0$.

\subsection{\textbf{\textsc{Decomposition of symbolic integral over width K right trapezium}}}
In this section we find (level two) $G$-decompositions for symbolic double integrals over right trapizoidal regions of expressions $q^{f(z,w)}$ where $f$ is a general quadratic. We do this in stages, first using symmetry for $f$ homogeneous and over a right triangular region whose diagonal passes through the origin, then shifting to a more general right trapezium and finally for a general inhomogeneous quadratic.

\begin{lem} Suppose $m\in\Q$. Then for the homogeneous part $f_h(z,w)=f_1z^2+f_2zw+f_3w^2$,
\[ \int_B ^{K+ B} \int_{B }^{mz+1 + B }q^{f_h(z,w) } dw dz
\ePK \frac{m}{2} \left( \frac{(\Delta)_K}{K} \right)\G^2  +   \frac12 \left( \frac{(f_1)_K}{K} \right)\G - \frac{m}{2} \left(\frac{(f_3)_K}{K} \right) \G+ \frac12 \left( \frac{(f_h(1,m))_K}{K} \right) \G\]
 so long as $K$ is an odd prime such that $m\in\ZK$, $f_h(z,w)\in\ZK[z,w]$ and $f_1,f_3,f_h(1,m),\Delta\neK0$.
\end{lem}
\begin{proof}
Applying twice the central inversion formula for symbolic integrals, $\int_{a+b}^{b+B}f(-x)dx=\int_{1-b+B}^{1-a+B}f(x)dx$,
\[
\int_B ^{K+ B} \int_{B }^{mz+1 + B }q^{f_h(-z,-w) } dw dz
=\int_B ^{K+ B} \int_{-mz+B }^{1 + B }q^{f_h(-z,w) } dw dz
=\int_{1-K+B} ^{1+ B} \int_{mz+B }^{1 + B }q^{f_h(z,w) } dw dz
\]
Changing variables by $z'=z+K$ in the right hand side and observing that $f_h(-z,-w)=f_h(z,w)$, we get
\[
\int_B ^{K+ B} \!\!\!\!\int_{B }^{mz+1 + B }\!\!\!\!q^{f_h(z,w) } dw dz
\!=\!\!\!\int_{1+B} ^{1+K+ B}\!\!\! \int_{mz'-mK+B }^{1 + B }\!q^{f_h(z'-K,w) } dw dz'
\!=\!\!\!\int_{1+B} ^{1+K+ B}\!\!\! \int_{mz'+B }^{1 +mK+ B }\!\!\!q^{f_h(z'-K,w'-mK) } dw' dz'
\]
where in the last step we translated the inner integral by $w'=w+mK$. Since $f_h(z,w)\in\ZK[z,w]$, thus $f_h(z-K,w-mK)\eK{}f_h(z,w)$ for all $z,w\in\Z$ and so by Lemma \ref{modKmodphiK}, $q^{f_h(z-K,w-mK)}\ePK{}q^{f_h(z,w)}$ in $\ZKh$ for all $z,w\in\Z$. By the two variable version of Lemma \ref{pzqmodPhicharacterisation}, thus $q^{f_h(z-K,w-mK)}\ePK{}q^{f_h(z,w)}$ in $\PzwZKh$. By Lemma \ref{Bintegralcongruence} applied twice it follows that
\[
\int_B ^{K+ B} \int_{B }^{mz+1 + B }q^{f_h(z,w) } dw dz
\ePK\int_{1+B} ^{1+K+ B} \int_{mz+B }^{1 +mK+ B }q^{f_h(z,w) } dw dz
\]
Applying $\int_{1+B}^{1+K+B}g(z)dz=\int_B^{K+B}g(z)dz-g(0)+g(K)$  to $g(z)=\int_{mz+B }^{1 +mK+ B }q^{f_h(z,w) } dw$ we get
\[
\int_B ^{K+ B} \int_{B }^{mz+1 + B }q^{f_h(z,w) } dw dz
\ePK\int_{B} ^{K+ B} \int_{mz+B }^{1 +mK+ B }q^{f_h(z,w) } dw dz-\int_{B }^{1 +mK+ B }q^{f_h(0,w) } dw
+q^{f_h(0,0) }
\]
By cutting up the integral and recalling that $\int_{x+B}^{x+1+B}h(z)dz=h(x)$, we can replace the inner integral on the RHS
\[\int_{mz+B }^{1 +mK+ B }q^{f_h(z,w) } dw
=\int_{mz+1+B }^{mK+ B }q^{f_h(z,w) } dw+q^{f_h(z,mz)}+q^{f_h(z,mK)}
\]
and similarly $\int_{B }^{1 +mK+ B }q^{f_h(0,w) } dw-q^{f_h(0,0)}\ePK\int_{B }^{mK+ B }q^{f_h(0,w) } dw$ giving finally
\[
\int_B ^{K+ B} \!\!\!\int_{B }^{mz+1 + B }\!\!\!q^{f_h(z,w) } dw dz
\ePK\!\!\int_{B} ^{K+ B} \!\!\!\int_{mz+1+B }^{mK+ B }\!q^{f_h(z,w) } dw dz+\!\int_B^{K+B}\!\!(q^{f_h(z,mz)}+q^{f_h(z,0)})dz-\!\int_{B }^{mK+ B }\!\!\!q^{f_h(0,w) } dw
\]
Adding the left hand side to both sides, we get
\[
2\!\int_B ^{K+ B} \!\!\!\!\int_{B }^{mz+1 + B }\!\!\!q^{f_h(z,w) } dw dz
\ePK\!\!\int_{B} ^{K+ B} \!\!\!\int_{B }^{mK+ B }\!q^{f_h(z,w) } dw dz+\!\int_B^{K+B}\!\!(q^{f_h(z,mz)}+q^{f_h(z,0)})dz-\!\int_{B }^{mK+ B }\!\!\!q^{f_h(0,w) } dw
\]
The first term on the right hand side is now a symbolic repeated integral over a rectangle and so the order of integration can be reversed without changing the result. The integrals $\int_B^{K+B}q^{f(z,w)}dz$, $\int_B^{K+B}q^{f_h(z,mz)}dz$ and $\int_B^{K+B}q^{f_h(z,0)}dz$ are all Gauss sums which can be evaluated by Lemmas \ref{generalquadraticsimplified} and \ref{Gaussslice} to give
\begin{align*}
2\int_B ^{K+ B} \int_{B }^{mz+1 + B }q^{f_h(z,w) } dw dz
\ePK\left( \frac{(f_1)_K}{K} \right)\G\int_{B }^{mK+ B }{}q^\frac{\Delta{}w^2}{4f_1} dw
&+\left( \frac{(f_h(1,m))_K}{K} \right) \G+\left( \frac{(f_1)_K}{K} \right)\G\\
&-\int_{B }^{mK+ B }q^{f_3w^2 } dw
\end{align*}
The result follows from Lemma \ref{1dcarry-overcorollary} which allows the integrals $\int_B^{mK+B}$ to be evaluated modulo $\Phi_K$ as $m$ times integrals $\int_B^{K+B}$, which are themselves Gauss sums whose values are given by Lemma \ref{easyGauss}.
\end{proof}

\begin{lem}\label{Krighttrianglehomogeneous} Suppose $m,c,d\in\Q$. Then
 \begin{align*}
     \int_{B} ^{K+ B}\!\! \int_{d+ B }^{mz+ c+1 + B }\!\!\!q^{f_h(z,w) } dw dz
     \ePK \frac{m}{2} &\left( \frac{(\Delta)_K}{K} \right)\G^2
     -\left(\frac{(f_1)_K}{K} \right)\G\left[\int_B ^{d+ B }  q^{\frac{ \Delta }{4f_1}w^2 }dw -\frac12\right]\\
     & - \frac{m}{2} \left(\frac{(f_3)_K}{K} \right) \G
     + \left(\frac{ (f_h(1,m))_K }{K} \right)\G\left[\int_{B} ^{c+1+ B } q^{ \frac{\Delta }{4f_h(1,m)}w^2 } dw -\frac12\right]
 \end{align*}
 so long as $K$ is an odd prime such that $m,c,d\in\ZK$, $f_h\in\ZK[z,w]$ and $f_h(1,0),f_h(0,1),f_h(1,m),\Delta\not\equiv_K0$.
\end{lem}
\begin{proof}
    Split the inner symbolic integral into three and then use translation
    \begin{align*}
&\int_B ^{K+ B} \int_{d+ B }^{mz+ c+1 + B }q^{f_h(z,w) } dw dz \\
&=\int_B^{K+B} \int_B^{mz+1+B}q^{f_h(z,w)}dwdz -\int_B^{K+B}\int_B^{d+B}q^{f_h(z,w)}dwdz
+\int_B^{K+B}\int_{mz+1+B}^{mz+1+c+B}q^{f_h(z,w)}dwdz\\
&=\int_B^{K+B} \int_B^{mz+1+B}q^{f_h(z,w)}dwdz -\int_B^{K+B}\int_B^{d+B}q^{f_h(z,w)}dwdz+\int_B^{K+B}\int_B^{c+B}q^{f_h(z,w+mz+1)}dwdz
    \end{align*}
The first integral is evaluated by the previous lemma. Since the second and third integrals are symbolic repeated integrals over rectangular domains, the order of integration can be reversed and then the new inner integrals are generalised Gauss sums which can be evaluated by Lemma \ref{Gaussslice} with $\bfm=(1,0)$, $\bfa=(0,w)$  for the second integral and  $\bfm=(1,m)$, $\bfa=(0,w+1)$ for the third integral
\[
\int_B^{K+B}q^{f_h(z,w)}dz\ePK\left(\frac{(f_h(1,0))_K}{K}\right)\G{}q^{\frac{\Delta{w}^2}{4f_h(1,0)}},\quad
\int_B^{K+B}q^{f_h(z,w+mz+1)}dz
\ePK\left(\frac{(f_h(1,m))_K}{K}\right)\G{}q^{\frac{\Delta(w+1)^2}{4f_h(1,m)}}
\]
Combining
 \begin{align*}
     \int_{B} ^{K+ B}\!\! \int_{d+ B }^{mz+ c+1 + B }\!\!\!q^{f_h(z,w) } dw dz
     \ePK \frac{m}{2} &\left( \frac{(\Delta)_K}{K} \right)\G^2  +   \frac12 \left( \frac{(f_1)_K}{K} \right)\G - \frac{m}{2} \left(\frac{(f_3)_K}{K} \right) \G + \frac12 \left( \frac{(f_h(1,m))_K}{K} \right) \G\\
     & -\left(\frac{(f_1)_K}{K} \right)\G\int_B ^{d+ B }  q^{\frac{ \Delta w^2 }{4f_1}}dw
     + \left(\frac{ (f_h(1,m))_K }{K} \right)\G\int_{B} ^{c+ B } q^{ \frac{\Delta (w+1)^2}{4f_h(1,m)} } dw
 \end{align*}
 which rearranged gives the required result after shifting the last integral by 1.
\end{proof}

\begin{lem}\label{plemma}
    Suppose that $K$ is an odd prime for which $f(z,w)\in\ZK[z,w]$ while $\Delta,f_h(0,1)\neK0$. Then for any $m,c\in\ZK$,
    \[\p(z;q):=\int_{B}^{mz+c+B}q^{f(z,w)}dw\]
    defines an element of $\P(z;\ZK)[[h]]$ satisfying
    $\p(z+K;q)\ePK\p(z;q)+m\left(\frac{(f_h(0,1))_K}{K}\right)\G{}q^{\frac{D}{\Delta}+\frac{\Delta}{4f_h(0,1)}(z-z_0)^2}$
\end{lem}
\begin{proof}
    Since $f(z,w)$ is a polynomial with coefficients in $\ZK$, thus $f(z+K,w)\eK{}f(z,w)$ in $\ZK[z,w]$. By Lemma \ref{qpowerfzw}, $q^{f(z+K,w)}\ePK{}q^{f(z,w)}$ in $\P(z,w;\ZK)[[h]]$. By Lemma \ref{Bintegralcongruence},
    \[
\p(z+K;q)\ePK\!\!\int_{B}^{mz+mK+c+B}\!\!q^{f(z,w)}dw=\p(z;q)+\int_{mz+c+B}^{mz+mK+c+B}\!\!q^{f(z,w)}dw
\ePK\p(z;q)+m\!\int_{B}^{K+B}\!\!q^{f(z,w)}dw
    \]
    using Lemma \ref{1dcarry-overcorollaryzw} in the last step. The result now follows by evaluating the last integral by Lemma \ref{generalquadraticsimplified}.
\end{proof}
\noindent As a corollary, by taking the difference of two such expressions as in the previous lemma, we see that any expression obtained by symbolic integration between limits of the form (linear function of $z$)$+B$ will yield a similar property, with the coefficient $m$ on the right hand side replaced by the difference in slopes coming from the upper and lower limits of integration.

\begin{lem} \label{nonhomogeneouscmonster}
For $a,b,m,c,d\in\Q$ and a general quadratic $f(z,w)\in\Q[z,w]$ with homogeneous part $f_h$,
\begin{align*}
    &\int_{a+B} ^{a+bK+ B }\!\!\!\! \int_{d+B} ^{mz+c +1 + B } \!\!\!q^{f(z,w) } dw dz
\ePK q^\frac{D}{\Delta} \Bigg[  \frac{b^2m}{2} \left(\frac{(\Delta)_K}{K} \right) \G^2
 \!-b \left( \frac{(f_h(1,0))_K}{K}  \right) \G \left[ \int_B^{d-w_0+B} \!\!\!q^{ \frac{\Delta w^2 }{4f_h(1,0)}} dw - \frac12 \right] \\
&+b\left(\frac{(f_h (1,m))_K}{K} \right)\G \left[ \int_{B} ^{c+mz_0-w_0+1 + B } \!\!q^{\frac{\Delta w^2}{4f_h (1,m)}} dw -\frac12 \right]
\!+ bm \left(\frac{(f_h(0,1))_K}{K} \right)\G \left[\int_B ^{a-z_0 + B } \!\!\!q^{\frac{ \Delta z^2 }{4f_h(0,1)}} dz  - \frac{1}{2} \right] \Bigg]
\end{align*}
for all odd primes $K$ for which $f(z,w)\!\in\ZK[z,w]$, $a,b,m,c,d\in\ZK$ with $\Delta,f_h(1,0),f_h(0,1),f_h(1,m)\neK0$.
       \end{lem}

\begin{proof} By lemma \ref{generalquadraticslice}, $f(z,w)=f_h(z-z_0,w-w_0)+\frac{D}{\Delta}$ so that
\[
\int_{a+B} ^{a+b K+ B  }\!\!\!\!\int_{d+B} ^{mz +c+ 1 + B }\!\!\! q^{f(z,w)} dw dz
=q^\frac{D}{\Delta}\!\!\int_{a-z_0+B} ^{ a+bK-z_0+ B  }\!\!\!\! \int_{d-w_0+B} ^{mz +mz_0+c+ 1-w_0 + B }\!\!\! q^{f_h(z,w)} dw dz
=\!q^\frac{D}{\Delta}\!\!\int_{a-z_0+B} ^{ a+bK-z_0+ B  }\!\!\!\! \p(z;q) dz
\]
where $\p(z;q)=\int_{d-w_0+B} ^{mz +mz_0+c+ 1-w_0 + B } q^{f_h(z,w)} dw$. By Lemma \ref{plemma},
$\p(z+K;q)-\p(z;q)\ePK{}m\left(\frac{(f_3)_K}{K}\right)\G{}q^{\frac{\Delta{}z^2}{4f_3}}$ and so applying Theorem \ref{generalcarry-over},
\[
LHS
=q^\frac{D}{\Delta}\left[b\int_{B} ^{ K+ B  }\p(z;q)dz+bm\left(\frac{(f_3)_K}{K}\right)\G\int_B^{a-z_0+B}q^{\frac{\Delta{}z^2}{4f_3}}dz+\frac{b(b-1)}{2}m\left(\frac{(\Delta)_K}{K}\right)\G^2\right]
\]
The first integral can be evaluated by Lemma \ref{Krighttrianglehomogeneous}
\begin{align*}
\int_{B} ^{ K+ B  }\p(z;q)dz\ePK \frac{m}{2} &\left( \frac{(\Delta)_K}{K} \right)\G^2  -\left(\frac{(f_1)_K}{K} \right)\G\left[\int_B ^{d-w_0+ B }  q^{\frac{ \Delta }{4f_1}w^2 }dw -\frac12\right]\\
     & - \frac{m}{2} \left(\frac{(f_3)_K}{K} \right) \G
     + \left(\frac{ (f_h(1,m))_K }{K} \right)\G\left[\int_{B} ^{mz_0+c-w_0+1+ B } q^{ \frac{\Delta }{4f_h(1,m)}w^2 } dw -\frac12\right]
\end{align*}
which, after rearranging terms,  completes the proof.
\end{proof}

\subsection{\textbf{\textsc{Sum over right trapezium with rational slope}}}

\begin{thm}\label{decompositionoftrapezium} Suppose that $\alpha,\beta,\gamma,m,c\in\Q$ with $\alpha<\beta$ and $\gamma-c\leq{}m\alpha,m\beta$ while $f_1,f_3,f_h(1,m)\not=0$. Consider the right trapezium $\trap: \alpha \leq z \leq \beta$, $\gamma\leq w \leq mz + c $. Then the sum
\[S_K (\trap) = \sum_{ (z,w) \in K\trap \cap \Z^2  } q^{f(z,w)}\] defines a squak which has a level two $\G$-decomposition, for all but finitely many primes $K$,
\begin{align*}
    S_K (\trap) \ePK q^\frac{D}{\Delta} \left( \frac{(\Delta)_K}{K} \right)  A (\trap)  \G^2 +  \left( Z_{K,\rightarrow} + Z_{K,\uparrow} + Z_{K,\swarrow} +Z_{K,\downarrow}  \right)\G + Z_K (\trap)
\end{align*}
    where $A(\trap)$ is the area of $\trap$ and  $Z_{K,\rightarrow}$, $Z_{K,\uparrow}$, $Z_{K,\swarrow}$, $Z_{K,\downarrow}$, $Z_K (\trap)$ are almost independent of $K$. The terms $Z_{K,\rightarrow}$, $Z_{K,\uparrow}$, $Z_{K,\swarrow}$, $Z_{K,\downarrow}$ are associated with the boundary of $\trap$ taken in an anticlockwise orientation, namely the bottom horizontal edge, right hand vertical edge, top diagonal edge and left hand vertical edge, respectively, in the sense that they depend only on the data of these particular segments and not on that of the rest of $\trap$.
\end{thm}
\begin{proof} Choose $K$ to be an odd prime such that $f(z,w)\in\ZK[z,w]$ with $\Delta,f_h(1,0),f_h(0,1),f_h(1,m)\neK0$ while $\alpha,\beta,\gamma, m ,c \in \Z_{(K)}$. Write $m = \frac{m_2}{m_1}$ in reduced form, $m_1\in\N$, $m_2\in\Z$ coprime; since $m \in \Z_{(K)}$ thus $(m_1, K)=1$. Since $f(z,w)\in\ZK$ for all $z,w\in\Z$, thus by Lemma \ref{qpoweralpha}, $q^{f(z,w)}\in\ZKh$ for all $z,w\in\Z$ and hence the sum $S_K(\trap)$ defines a squak. Split the domain of summation according to the residue of $z$ modulo $m_1$,
\begin{align*}
 K\trap\cap\Z^2
 &=\big\{(z,w)\in\Z^2\big|\,K\alpha\leq{}z\leq{}K\beta,\quad{}K\gamma\leq{}w\leq{}mz+Kc\big\}\\
 &= \bigcup\limits_{r=0}^{m_1-1}\!\bigg\{(m_1z+r,w) \bigg|\, (z,w)\in\Z^2, \,
 \left\lceil\frac{K\alpha-r}{m_1}\right\rceil\leq z\leq\left\lfloor\frac{K\beta-r}{m_1}\right\rfloor,\,
 \lceil{}K\gamma\rceil\leq{}w\leq{}m_2z+\lfloor{}mr+Kc\rfloor\bigg\}
\end{align*}
Then
\[
 S_K (\trap)
 =\sum_{r=0}^{m_1-1}\left(\sum_{z=a_r}^{b_r}
 \sum_{w=\lceil{}K\gamma\rceil}^{m_2z+\lfloor{}mr+Kc\rfloor}q^{f(m_1z+r,w)}\right)
 = \sum_{r=0}^{m_1 -1 }  \left( \int_{ a_r + B }^{  b_r+1  + B   }  \int_{ \lceil K\gamma \rceil+B  }^{ m_2 z +  \lfloor    mr+ Kc  \rfloor  +1 + B }  q^{f(m_1 z+r,w)} dw dz \right).
\]
where $a_r=\left\lceil\frac{K\alpha-r}{m_1}\right\rceil$ and $b_r=\left\lfloor\frac{K\beta-r}{m_1}\right\rfloor$.
We now split the inner and outer integrals in order to obtain limits of symbolic integration which are almost independent of $K$, starting with the inner integral. Set $c_r=\lfloor{}mr+Kc\rfloor-Kc$  and note that $\lceil{}K\gamma\rceil-K\gamma=\{-K\gamma\}$. Then splitting the inner integral and using Lemma \ref{1dcarry-overcorollary} on the intervals
$[\{-K\gamma\}+B,\lceil{}K\gamma\rceil+B]$ and $[m_2z+c_r+1+B,m_2 z +  \lfloor    mr+ Kc  \rfloor  +1 + B]$ of lengths $K\gamma$, $Kc$ respectively, gives
\begin{align*}
 \int_{ \lceil K\gamma \rceil+B  }^{ m_2 z +  \lfloor    mr+ Kc  \rfloor  +1 + B }  q^{f(m_1 z+r,w)} dw
 &\ePK \int_{ \{-K\gamma\}+B  }^{ m_2 z +  c_r  +1 + B }  q^{f(m_1 z+r,w)} dw +(c-\gamma)\int_B^{K+B} q^{f(m_1 z+r,w)} dw \\
 &\ePK\int_{ \{-K\gamma\}+B  }^{ m_2 z +  c_r  +1 + B }  q^{f(m_1 z+r,w)} dw +(c-\gamma)\left(\frac{(f_3)_K}{K}\right)\G{}q^{\frac{D}{\Delta}+\frac{\Delta(m_1z+r-z_0)^2}{4f_3}}
\end{align*}
using Lemma \ref{Gaussslice}. We now have
\[
S_K(\trap)
=\sum_{r=0}^{m_1 -1 } \!\!\!\left( \int_{ a_r + B }^{  b_r+1  + B   }\!\!\!  \int_{ \{-K\gamma\}+B  }^{ m_2 z + c_r  +1 + B } \!\!\! q^{f(m_1 z+r,w)} dw dz \right)
+(c-\gamma)\left(\frac{(f_3)_K}{K}\right)\G\!\!\sum_{r=0}^{m_1 -1 }\int_{ a_r + B }^{  b_r+1  + B   }\!\!q^{\frac{D}{\Delta}+\frac{\Delta(m_1z+r-z_0)^2}{4f_3}}dz
\]
Set $x_r= a_r - \frac{K\alpha}{m_1}$ and $y_r= b_r-\frac{K\beta }{m_1}$; they are almost independent of $K$.
Split the outer integral in the summand into three by $x_r+B$ and $y_r+1+B$. Lemma \ref{nonhomogeneouscmonster} applies on the segments of lengths divisible by $K$ in $\ZKh$. In particular, for the segment from $x_r+B$ to $a_r+B$ apply Lemma \ref{nonhomogeneouscmonster} with $a,b,c,d,m$ replaced by $x_r,\frac{\alpha}{m_1},c_r,\{-K\gamma\},m_2$ respectively and $f(z,w)$ replaced by $f(m_1z+r,w)$, noting that $f_1,f_3,f_h(1,m_2),\Delta,D,z_0,w_0$ have become $m_1^2f_1,f_3,f_h(m_1,m_2),m_1^2\Delta,m_1^2{}D,\frac{z_0-r}{m_1},w_0$ respectively, to give
\begin{align*}
 &\int_{ x_r + B }^{  a_r  + B   }  \int_{ \{- K\gamma \}+B  }^{ m_2 z + c_r  +1 + B }  q^{f(m_1 z+r,w)} dw dz   \\
& \ePK q^\frac{D}{\Delta} \Bigg[  \frac{\alpha^2m_2}{2m_1^2} \left(\frac{(\Delta)_K}{K} \right) \G^2
 +\frac{\alpha}{m_1}\left(\frac{(f_h (m_1,m_2))_K}{K} \right)\G \left\{ \int_{B} ^{c_r+mz_0-mr-w_0+1 + B } q^{\frac{m_1^2\Delta w^2}{4f_h (m_1,m_2)}} dw -\frac12 \right\}  \\
&\,-\frac{\alpha}{m_1} \left( \frac{(f_1)_K}{K}  \right) \G \left\{ \int_B^{\{-K\gamma\}-w_0+B} q^{ \frac{\Delta w^2 }{4f_1}} dw - \frac12 \right\}
+ \frac{\alpha{}m_2}{m_1} \left(\frac{(f_3)_K}{K} \right)\G \left\{\int_B ^{x_r-\frac{z_0-r}{m_1} + B } q^{\frac{ m_1^2\Delta z^2 }{4f_3}} dz  - \frac{1}{2} \right\} \Bigg]
\end{align*}
Similarly on the segment from $y_r+1+B$ to $b_r+1+B$ we get
\begin{align*}
 &\int_{ y_r +1+ B }^{  b_r + 1+B   }  \int_{ \{- K\gamma\}+B  }^{ m_2 z + c_r  +1 + B }  q^{f(m_1 z+r,w)} dw dz   \\
& \ePK q^\frac{D}{\Delta} \Bigg[  \frac{\beta^2m_2}{2m_1^2} \left(\frac{(\Delta)_K}{K} \right) \G^2
 +\frac{\beta}{m_1}\left(\frac{(f_h (m_1,m_2))_K}{K} \right)\G \left\{ \int_{B} ^{c_r+mz_0-mr-w_0+1 + B } q^{\frac{m_1^2\Delta w^2}{4f_h (m_1,m_2)}} dw -\frac12 \right\}  \\
&\,-\frac{\beta}{m_1} \left( \frac{(f_1)_K}{K}  \right) \G \left\{ \int_B^{\{-K\gamma\}-w_0+B} q^{ \frac{\Delta z^2 }{4f_1}} dz - \frac12 \right\}
+ \frac{\beta{}m_2}{m_1} \left(\frac{(f_3)_K}{K} \right)\G \left\{\int_B ^{y_r+1-\frac{z_0-r}{m_1} + B } q^{\frac{ m_1^2\Delta z^2 }{4f_3}} dz  - \frac{1}{2} \right\} \Bigg]
\end{align*}
The final segment from $x_r+B$ to $y_r+1+B$ is almost independent of $K$. We also split the single integral in the same manner using Lemmas \ref{1dcarry-overcorollary} and \ref{easyGauss}
\[
\int_{ a_r + B }^{  b_r+1  + B   }\!\!q^{\frac{\Delta(m_1z+r-z_0)^2}{4f_3}}dz
=\int_{ x_r + B }^{  y_r+1  + B   }\!\!q^{\frac{\Delta(m_1z+r-z_0)^2}{4f_3}}dz
+\frac{\beta-\alpha}{m_1}\left(\frac{(\Delta/f_3)_K}{K}\right)\G
\]
Combining,
\begin{align*}
S_K (\trap)
 &= \left[\frac12m(\beta-\alpha)^2+(c-\gamma)(\beta-\alpha)\right]\!\!\left(\frac{(\Delta)_K}{K}\right)\G^2q^\frac{D}{\Delta}+\!\!\sum_{r=0}^{m_1 -1 }  \bigg[ \int_{ x_r + B }^{  y_r+1  + B   }  \int_{ \{- K\gamma\}+B }^{ m_2 z +c_r+1+B }  q^{f(m_1 z+r,w)} dw dz  \bigg]\\
 &+q^\frac{D}{\Delta}\G
 \frac{\beta-\alpha}{m_1}\left(\frac{(f_h(m_1,m_2))_K}{K}\right)
 \sum_{r=0}^{m_1-1}\left[\int_B^{c_r+mz_0-mr-w_0+1+B}q^\frac{m_1^2\Delta{}w^2}{4f_h(m_1,m_2)}dw-\frac12\right]\\
 &-q^\frac{D}{\Delta}\G\frac{\beta-\alpha}{m_1}\left(\frac{(f_1)_K}{K}\right)\sum_{r=0}^{m_1-1}\left[\int_B^{\{-K\gamma\}-w_0+B}q^\frac{\Delta{}w^2}{4f_1}dw-\frac12\right]\\
 &+\left(\frac{(f_3)_K}{K}\right)q^\frac{D}{\Delta}\G\sum_{r=0}^{m_1-1}\bigg(m\beta\left[\int_B^{y_r+1-\frac{z_0-r}{m_1}+B}q^\frac{m_1^2\Delta{}z^2}{4f_3}dz-\frac12\right]
 -m\alpha\left[\int_B^{x_r-\frac{z_0-r}{m_1}+B}q^\frac{m_1^2\Delta{}z^2}{4f_3}dz-\frac12\right]\\
 &\hspace{150pt}+(c-\gamma)\int_{ x_r + B }^{  y_r+1  + B   }\!\!q^{\frac{\Delta(m_1z+r-z_0)^2}{4f_3}}dz\bigg)
 \end{align*}
which is a level two $\G$ decomposition in the form of the theorem, where $A(\trap)=\frac12m(\beta-\alpha)^2+(c-\gamma)(\beta-\alpha)$ is the area of the trapezium and
\begin{align}
Z_K(\trap)&=\sum_{r=0}^{m_1 -1 }  \bigg( \int_{ x_r + B }^{  y_r+1  + B   }  \int_{ \{- K\gamma\}+B }^{ m_2 z +c_r+1+B }  q^{f(m_1 z+r,w)} dw dz  \bigg)\\
Z_{K,\rightarrow}&=-q^\frac{D}{\Delta}\left(\frac{(f_1)_K}{K}\right)(\beta-\alpha)\left[\int_B^{\{-K\gamma\}-w_0+B}q^\frac{\Delta{}w^2}{4 f_1}dw-\frac12\right]\\
Z_{K,\uparrow}&=q^\frac{D}{\Delta}\left(\frac{(f_3)_K}{K}\right)(m\beta+c-\gamma)
\sum_{r=0}^{m_1-1}\left[\int_B^{y_r+1+\frac{r-z_0}{m_1}+B}q^\frac{m_1^2\Delta{}z^2}{4f_3}dz-\frac12\right]\notag\\
Z_{K,\swarrow}&= q^\frac{D}{\Delta}\left(\frac{(f_h(m_1,m_2))_K}{K}\right)\frac{\beta-\alpha}{m_1}
 \sum_{r=0}^{m_1-1}\left[\int_B^{c_r+mz_0-mr-w_0+1+B}q^\frac{m_1^2\Delta{}w^2}{4f_h(m_1,m_2)}dw-\frac12\right]\notag \\
Z_{K,\downarrow}&=-q^\frac{D}{\Delta}\left(\frac{(f_3)_K}{K}\right)(m\alpha+c-\gamma)
\sum_{r=0}^{m_1-1}\left[\int_B^{x_r+\frac{r-z_0}{m_1}+B}q^\frac{m_1^2\Delta{}z^2}{4f_3}dz-\frac12\right]\notag
\end{align}
where we have translated and split the integral in the last line for $S_K(\trap)$ above,
\[
\int_{ x_r + B }^{  y_r+1  + B   }\!\!q^{\frac{\Delta(m_1z+r-z_0)^2}{4f_3}}dz
=\int_{ x_r+\frac{r-z_0}{m_1} + B }^{  y_r+1+\frac{r-z_0}{m_1} + B   }\!\!q^{\frac{m_1^2\Delta{}z^2}{4f_3}}dz
=\int_{B }^{  y_r+1+\frac{r-z_0}{m_1} + B   }\!\!q^{\frac{m_1^2\Delta{}z^2}{4f_3}}dz
-\int_{B }^{  x_r+\frac{r-z_0}{m_1} + B   }\!\!q^{\frac{m_1^2\Delta{}z^2}{4f_3}}dz
\]
which now contributes to both $Z_{K,\uparrow}$ and $Z_{K,\downarrow}$.
Additionally, the fourth term in the above expression for $S_K(\trap)$, which was identified as $Z_{K,\rightarrow}$, contains a sum over $r$ whose summand is independent of $r$ and thus becomes $m_1$ times this constant summand. We now proceed to simplify the remaining expressions, starting with the vertical edge contributions.
Recall that $y_r=\left\lfloor\frac{K\beta-r}{m_1}\right\rfloor-\frac{K\beta }{m_1}$
and so $y_r+\frac{r}{m_1}=-\left\{\frac{K\beta-r}{m_1}\right\}$.
The expression $\left\{\frac{K\beta-r}{m_1}\right\}$ as a function of $r$ is periodic with period $m_1$
and its smallest value is $\frac{\{K\beta\}}{m_1}$;
thus its range is $\Big\{\frac{\{K\beta\}+s}{m_1}\Big|s=0,1,\ldots,m_1-1\Big\}$. Hence
\begin{align*}
\sum_{r=0}^{m_1-1}\left[\int_B^{y_r+1+\frac{r-z_0}{m_1}+B}q^\frac{m_1^2\Delta{}z^2}{4f_3}dz-\frac12\right]
&=\sum_{s=0}^{m_1-1}\left[\int_B^{1-\frac{\{K\beta\}+s+z_0}{m_1}+B}q^\frac{m_1^2\Delta{}z^2}{4f_3}dz-\frac12\right]\\
&=-\sum_{s=0}^{m_1-1}\left[\int_B^{\frac{\{K\beta\}+s+z_0}{m_1}+B}q^\frac{m_1^2\Delta{}z^2}{4f_3}dz-\frac12\right]
\hbox{, by inversion Lemma \ref{uniquenessinarbitrarydimension}}\\
&=-\left[\int_B^{\{K\beta\}+z_0+B}q^\frac{\Delta{}z^2}{4f_3}dz-\frac12\right]
\hbox{, by Theorem \ref{Bernoullimultiplicationlemma3}}
\end{align*}
Similarly $x_r=\left\lceil\frac{K\alpha-r}{m_1}\right\rceil- \frac{K\alpha}{m_1}$ and so
$x_r+\frac{r}{m_1}=\left\{\frac{r-K\alpha}{m_1}\right\}$. This is also periodic in $r$ of period $m_1$ and takes values $\frac{s+\{-K\alpha\}}{m_1}$ for $s=0,1,\ldots,m_1-1$. Thus
\[
\sum_{r=0}^{m_1-1}\left[\int_B^{x_r+\frac{r-z_0}{m_1}+B}q^\frac{m_1^2\Delta{}z^2}{4f_3}dz-\frac12\right]
=\sum_{s=0}^{m_1-1}\left[\int_B^{\frac{s+\{-K\alpha\}-z_0}{m_1}+B}q^\frac{m_1^2\Delta{}z^2}{4f_3}dz-\frac12\right]
=\int_B^{\{-K\alpha\}-z_0+B}q^\frac{\Delta{}z^2}{4f_3}dz-\frac12
\]
by the Bernoulli multiplication Theorem \ref{Bernoullimultiplicationlemma3}. Thus the simplified contributions from the vertical edges are
\begin{align}
Z_{K,\uparrow}
&=q^\frac{D}{\Delta}\left(\frac{(f_3)_K}{K}\right)\lambda(\uparrow)
\left[\int_B^{1-\{K\beta\}-z_0+B}q^\frac{\Delta{}z^2}{4f_3}dz-\frac12\right]\\
Z_{K,\downarrow}&=-q^\frac{D}{\Delta}\left(\frac{(f_3)_K}{K}\right)\lambda(\downarrow)
\left[\int_B^{\{-K\alpha\}-z_0+B}q^\frac{\Delta{}z^2}{4f_3}dz-\frac12\right]
\end{align}
Finally for the diagonal term, recall that
$c_r=\lfloor{}mr+Kc\rfloor-Kc$ so that $c_r-mr=-\{mr+Kc\}$. Since $m_1$, $m_2$ are coprime, $mr=\frac{m_2r}{m_1}$ modulo 1 takes on all multiples of $\frac1{m_1}$ as $r$ ranges over the integers. Thus $\{mr+Kc\}$ is a periodic function of $r$ with period $m_1$ which has smallest value $\frac{\{Kcm_1\}}{m_1}$, that is its values are $\frac{\{Kcm_1\}+s}{m_1}$ for $s=0,1,\ldots,m-1$. Hence
\begin{align*}
\sum_{r=0}^{m_1-1}\left[\int_B^{c_r+mz_0-mr-w_0+1+B}q^\frac{m_1^2\Delta{}w^2}{4f_h(m_1,m_2)}dw-\frac12\right]
&=\sum_{s=0}^{m_1-1}\left[\int_B^{mz_0-w_0-\frac{\{Kcm_1\}+s}{m_1}+1+B}q^\frac{m_1^2\Delta{}w^2}{4f_h(m_1,m_2)}dw-\frac12\right]\\
&=-\sum_{s=0}^{m_1-1}\left[\int_B^{w_0-mz_0+\frac{\{Kcm_1\}+s}{m_1}+B}q^\frac{m_1^2\Delta{}w^2}{4f_h(m_1,m_2)}dw-\frac12\right]\\
&=-\left[\int_B^{m_1w_0-m_2z_0+\{Kcm_1\}+B}q^\frac{\Delta{}w^2}{4f_h(m_1,m_2)}dw-\frac12\right]
\end{align*}
using the inversion lemma \ref{symbolicinttranslationinversion} and Bernoulli multiplication theorem \ref{Bernoullimultiplicationlemma3} finally giving
\[
Z_{K,\swarrow}=q^\frac{D}{\Delta}\left(\frac{(f_h(m_1,m_2))_K}{K}\right)\lambda(\swarrow)
\left[\int_B^{m_2z_0-m_1w_0-\{Kcm_1\}+1+B}q^\frac{\Delta{}w^2}{4f_h(m_1,m_2)}dw-\frac12\right]
 \]
\end{proof}

\noindent {\bf Remark:} For the case of a rectangle, $m_1=1$, $m_2=0$ so that $a_0=\lceil{}K\alpha\rceil$, $b_0=\lfloor{}K\beta\rfloor$, $c_0=\lfloor{}Kc\rfloor-Kc=-\{Kc\}$ while $x_0=a_0-K\alpha=\{-K\alpha\}$, $y_0=b_0-K\beta=-\{K\beta\}$. The formula for the $\G^0$ term for the rectangle bounded by $z=\alpha,\beta$ and $w=\gamma,c$ becomes
\begin{align}
Z_K(\squarefilled)=\int_{\{-K\alpha\}+B}^{1-\{K\beta\}+B}\int_{\{-K\gamma\}+B}^{1-\{Kc\}+B}q^{f(z,w)}dwdz\label{Z0square}
\end{align}

\noindent For a downward pointing right trapezium a similar result follows by central inversion.

\begin{thm}\label{downwardtrapezium} Suppose that $\alpha,\beta,\gamma,m,c\in\Q$ with $\alpha<\beta$ and $\gamma-c\geq{}m\alpha,m\beta$ with $f_1,f_3,f(1.m)\not=0$. Consider the downward pointing right trapezium $\dtrap: \alpha \leq z \leq \beta$, $mz+c\leq w \leq \gamma$. Then the sum
\[S_K (\dtrap) = \sum_{ (z,w) \in K\dtrap \cap \Z^2  } q^{f(z,w)}\] defines a squak which has a level two $\G$-decomposition, for all but finitely many prime $K$,
\begin{align*}
    S_K (\dtrap) \ePK q^\frac{D}{\Delta} \left( \frac{(\Delta)_K}{K} \right)  A (\dtrap)  \G^2 +  \left( Z_{K,\leftarrow} + Z_{K,\downarrow} + Z_{K,\searrow} +\Z_{K,\uparrow}  \right)\G + Z_K (\dtrap)
\end{align*}
    where $A(\dtrap)$ is the area of $\dtrap$ and  $Z_{K,\leftarrow}$, $Z_{K,\downarrow}$, $Z_{K,\searrow}$, $Z_{K,\uparrow}$, $Z_K (\dtrap)$ are almost independent of $K$. The terms $Z_{K,\leftarrow}$, $Z_{K,\downarrow}$, $Z_{K,\searrow}$, $Z_{K,\uparrow}$ are associated with the boundary of $\dtrap$ taken in an anticlockwise orientation, namely the top horizontal edge, left hand vertical edge, bottom diagonal edge and right hand vertical edge, respectively.
\end{thm}
\begin{proof}
    Define $\fbar(z,w)=f(z,-w)$. Then $S_K(\dtrap)$ can be considered as a sum over an upward pointing trapezium defined by $\trap:\alpha \leq z \leq \beta$, $-\gamma\leq w \leq -mz-c$,
    \[
S_K(\dtrap)=\sum_{ (z,w) \in K\trap \cap \Z^2  } q^{\fbar(z,w)}
    \]
    Apply Theorem \ref{decompositionoftrapezium} with $m,c,\gamma$ replaced by their negations and $f$ replaced by $\fbar$. Then $m_2,f_2,f_5,w_0$ change sign while the rest of the parameters stay fixed, including $\fbar_h(m_1,-m_2)=f_h(m_1,m_2)$. Hence the formulae for $Z_{K,\uparrow}$ and $Z_{K,\downarrow}$ stay unchanged while
\begin{align*}
Z_K(\dtrap)&=\sum_{r=0}^{m_1 -1 }  \bigg( \int_{ x_r + B }^{  y_r+1  + B   }  \int_{ \{ K\gamma\}+B }^{ -m_2 z +\cbar_r+1+B }  q^{\fbar(m_1 z+r,w)} dw dz  \bigg)\\
Z_{K,\leftarrow}&=-q^\frac{D}{\Delta}\left(\frac{(f_1)_K}{K}\right)\lambda(\leftarrow)\left[\int_B^{\{K\gamma\}+w_0+B}q^\frac{\Delta{}w^2}{4 f_1}dw-\frac12\right]\\
Z_{K,\searrow}&= q^\frac{D}{\Delta}\left(\frac{(f_h(m_1,m_2))_K}{K}\right)\lambda(\searrow)
 \left[\int_B^{-m_2z_0+m_1w_0-\{-Kcm_1\}+1+B}q^\frac{\Delta{}w^2}{4f_h(m_1,m_2)}dw-\frac12\right]
\end{align*}
Applying the inversion Lemma \ref{symbolicinttranslationinversion} gives
\begin{align}
Z_K(\dtrap)&=\sum_{r=0}^{m_1 -1 }  \bigg( \int_{ x_r + B }^{  y_r+1  + B   }  \int_{ m_2z-\cbar_r+B }^{1-\{K\gamma\}+B }  q^{f(m_1 z+r,w)} dw dz  \bigg)\\
Z_{K,\leftarrow}&=q^\frac{D}{\Delta}\left(\frac{(f_1)_K}{K}\right)\lambda(\leftarrow)
\left[\int_B^{1-\{K\gamma\}-w_0+B}q^\frac{\Delta{}w^2}{4 f_1}dw-\frac12\right]\\
Z_{K,\searrow}&=-q^\frac{D}{\Delta}\left(\frac{(f_h(m_1,m_2))_K}{K}\right)\lambda(\searrow)
 \left[\int_B^{m_2z_0-m_1w_0+\{-Kcm_1\}+B}q^\frac{\Delta{}w^2}{4f_h(m_1,m_2)}dw-\frac12\right]
\end{align}
\end{proof}

\subsection{\textbf{\textsc{Oriented edge contributions}}}

We see that the formulae of Theorems \ref{decompositionoftrapezium} and \ref{downwardtrapezium} involve expressions for edge contributions which depend only on $f(z,w)$ and the data of the edge itself and not on data from other parts of the two-dimensional region over which we are summing. We collect here the formulae for horizontal edges at height $w=\gamma$ and  vertical edges at position $z=\alpha$ in the two orientations,
\begin{align}
Z_{K,\rightarrow}&=-q^\frac{D}{\Delta}\left(\frac{(f_1)_K}{K}\right)\lambda(\rightarrow)\left[\int_B^{\{-K\gamma\}-w_0+B}q^\frac{\Delta{}w^2}{4 f_1}dw-\frac12\right]\label{Zright}\\
Z_{K,\leftarrow}&=q^\frac{D}{\Delta}\left(\frac{(f_1)_K}{K}\right)\lambda(\leftarrow)
\left[\int_B^{1-\{K\gamma\}-w_0+B}q^\frac{\Delta{}w^2}{4 f_1}dw-\frac12\right]\label{Zleft}\\
Z_{K,\uparrow}
  &=q^\frac{D}{\Delta}\left(\frac{(f_3)_K}{K}\right)\lambda(\uparrow)
\left[\int_B^{1-\{K\alpha\}-z_0+B}q^\frac{\Delta{}z^2}{4f_3}dz-\frac12\right]\label{Zup}\\
Z_{K,\downarrow}&=-q^\frac{D}{\Delta}\left(\frac{(f_3)_K}{K}\right)\lambda(\downarrow)
\left[\int_B^{\{-K\alpha\}-z_0+B}q^\frac{\Delta{}z^2}{4f_3}dz-\frac12\right]\label{Zdown}
\end{align}
Observe that
\begin{align*}
    Z_{K,\rightarrow}+Z_{K,\leftarrow}&=q^\frac{D}{\Delta}\left(\frac{(f_1)_K}{K}\right)\lambda(-)
\int_{\{-K\gamma\}-w_0+B}^{1-\{K\gamma\}-w_0+B}q^\frac{\Delta{}w^2}{4 f_1}dw
=q^{\frac{D}{\Delta}+\frac{\Delta{}w_0^2}{4f_1}}\left(\frac{(f_1)_K}{K}\right)\lambda(-)\delta_{K\gamma\in\Z}
=Z^{(1)}_K(-)\\
Z_{K,\uparrow}+Z_{K,\downarrow}&=q^\frac{D}{\Delta}\left(\frac{(f_3)_K}{K}\right)\lambda(|)
\int_{\{-K\alpha\}-z_0+B}^{1-\{K\alpha\}-z_0+B}q^\frac{\Delta{}z^2}{4f_3}dz
=q^{\frac{D}{\Delta}+\frac{\Delta{}z_0^2}{4f_3}}\left(\frac{(f_1)_K}{K}\right)\lambda(|)\delta_{K\alpha\in\Z}
=Z^{(1)}_K(|)
\end{align*}
by Lemma \ref{BIntegralProperties}(1), which match the $\G$ coefficient of the corresponding line segment sums $S_K(-)$ and $S_K(|)$ in Theorem \ref{lineinvariant}.
For diagonal edges $w=mz+c$ where $\begin{pmatrix}m_1\\ m_2\end{pmatrix}$ is the reduced direction vector of the segment with $m_1\in\N$ and $m_2\in\Z$ coprime,
\begin{align}
Z_{K,\nearrow}&=-q^\frac{D}{\Delta}\left(\frac{(f_h(m_1,m_2))_K}{K}\right)\lambda(\nearrow)
 \left[\int_B^{m_2z_0-m_1w_0+\{-Kcm_1\}+B}q^\frac{\Delta{}w^2}{4f_h(m_1,m_2)}dw-\frac12\right]\label{Zne}\\
 Z_{K,\swarrow}&=q^\frac{D}{\Delta}\left(\frac{(f_h(m_1,m_2))_K}{K}\right)\lambda(\swarrow)
\left[\int_B^{m_2z_0-m_1w_0-\{Kcm_1\}+1+B}q^\frac{\Delta{}w^2}{4f_h(m_1,m_2)}dw-\frac12\right]\label{Zsw}
\end{align}
and so
\begin{align*}
Z_{K,\swarrow}+Z_{K,\nearrow}&=q^\frac{D}{\Delta}\left(\frac{(f_h(m_1,m_2))_K}{K}\right)\lambda(\cL)
\int_{m_2z_0-m_1w_0+\{-Kcm_1\}+B}^{m_2z_0-m_1w_0+\{-Kcm_1\}+1+B}q^\frac{\Delta{}w^2}{4f_h(m_1,m_2)}dw\\
&=q^{\frac{D}{\Delta}+\frac{\Delta(m_2z_0-m_1w_0)^2}{4f_h(m_1,m_2)}}\left(\frac{(f_h(m_1,m_2))_K}{K}\right)\lambda(\cL)\delta_{Km_1c\in\Z}=Z^{(1)}_K(\cL)
\end{align*}
by Lemma \ref{BIntegralProperties}(1), which again matches the coefficient of $\G$ in the level one $\G$ decomposition of $S_K(\cL)$ in Theorem \ref{lineinvariant}.

\noindent Note that \eqref{Zne} and \eqref{Zsw} specialise to \eqref{Zright} and \eqref{Zleft} when $\bfm=(\pm1,0)$.

\subsection{\textbf{\textsc{Extension to Arbitrary Polygons}}}

\begin{thm}\label{2DquantumPick} \textbf{\textit{(Conjecture \ref{quantumpickconjecture} in 2D)}}  Let $f(z,w)$ be a quadratic over $\Q$ with non-degenerate homogeneous part $f_h$ with $f_1,f_3\not=0$ and $\mathcal{P}$ be a rational simple polygon whose boundary edges are not parallel to null directions for $f_h$. Then the sum $S_K (\mathcal{P}) = \sum_{ K\mathcal{P} \cap \Z^2   }  q^{f(z,w)} $ defines a squak with a level two $\G$-decomposition
\begin{align*}
    S_K (\mathcal{P}) \equiv_{\Phi_K} \sum_{r=0}^2  Z_K^{(r)} (\mathcal{P}) \cdot \G^r
\end{align*}
where the coefficients $Z_K ^{(r)} (\cP)\in\ZKh$ are almost $K$-independent squaks. Furthermore,
\[
Z_K ^{(2)}(\cP) = A(\cP)q^{\frac{D}{\Delta}}  \left(\frac{(\Delta)_K}{K}\right),
\qquad{}Z_K^{(1)}(\cP) = \sum_{e\in\partial\cP} Z_{K,e}
\]
where $A(\cP)$ is the area of $\cP$, $Z_{K,e}$ is the oriented edge contribution defined by \eqref{Zup}, \eqref{Zdown}, \eqref{Zne}, \eqref{Zsw} and $\Delta$ is the discriminant of the homogeneous part of $f$.
\end{thm}
\begin{proof}
The proof proceeds by gluing to establish the result on progressively more general regions.
By Theorems \ref{decompositionoftrapezium} and \ref{downwardtrapezium}, the result holds for right trapezia. Next we deal with a more general trapezium, $\cT$ defined by $\alpha\leq{}z\leq\beta$ and $mz+c\leq{}w\leq{}m'z+c'$. Choose $\gamma$ sufficiently small that $\gamma<m\alpha+c,m\beta+c$. Consider the two upward pointing trapezia
\[
\cT_1:\alpha\leq{}z\leq\beta, \gamma\leq{}w\leq{}mz+c,\qquad
\cT_2:\alpha\leq{}z\leq\beta, \gamma\leq{}w\leq{}m'z+c'
\]
and let $\cL$ denote the edge $\alpha\leq{}z\leq\beta$, $w=mz+c$. Then
\[
S_K(\cT)=S_K(\cT_2)-S_K(\cT_1)+S_K(\cL)
\]
Applying Theorems \ref{decompositionoftrapezium}, \ref{downwardtrapezium} and \ref{lineinvariant} to the three terms we obtain level two $\G$ decompositions for the first two terms and a level one $\G$ decomposition for the third term, which combine to give a level two $\G$ decomposition of the combination,
\[
S_K(\cT)\ePK{}Z^{(2)}_K(\cT)\G^2+Z^{(1)}_K(\cT)\G+Z^{(0)}_K(\cT)
\]
where
\[
Z^{(2)}_K(\cT)=Z^{(2)}_K(\cT_2)-Z^{(2)}_K(\cT_1)
=q^\frac{D}{\Delta} \left( \frac{(\Delta)_K}{K} \right)  \big(A (\cT_2)-A(\cT_1)\big)
=q^\frac{D}{\Delta} \left( \frac{(\Delta)_K}{K} \right) A(\cT)
\]
since the areas are additive $A(\cT)=A(\cT_2)-A(\cT_1)$. The $\G$ term is
\[
Z^{(1)}_K(\cT)=Z^{(1)}_K(\cT_2)-Z^{(1)}_K(\cT_1)+Z^{(1)}_K(\cL)
\]
where $Z^{(1)}_K(\cT_1)$, $Z^{(1)}_K(\cT_2)$ each split as sums of four contributions from the respective boundary edges of $\cT_1$ and $\cT_2$. The vertical edge contributions from $\cT_1$ and $\cT_2$ are additive and combine to give the corresponding vertical edge contributions from $\cT$, since by \eqref{Zup} and \eqref{Zdown}, they are proportional to the lengths of the segments, which are additive. The horizontal edge contribution cancels. The upper slanting contribution is immediately what it should be while the lower slanting contribution works because
\[
-Z_{K,\swarrow}+Z^{(1)}(\cL)=Z_{K,\nearrow}
\]
as noted in the previous section. The term without $\G$ is
\[
Z^{(0)}_K(\cT)=Z^{(0)}_K(\cT_2)-Z^{(0)}_K(\cT_1)+Z^{(0)}_K(\cL)
\]
which is a combination of integrals, two double symbolic integrals and one ordinary symbolic integral, all of which are almost independent of $K$. A special case of this more general trapezium, when the left or right vertical edge has zero length, is a general triangle.

Finally for a general rational simple polygon $\cP$, one can slice it by vertical lines $z=z_i$ passing through the vertices of $\cP$, into trapezia $\cT_i$ of the type just dealt with (possibly also degenerate trapezia which are triangles). Using the result on each individual trapezium, summing and subtracting the overcounting on the internal vertical segments (on which Theorem \ref{lineinvariant} is applied) and possibly additionally double counted vertices, yields a level two $\G$ decomposition of $S_K(\cP)$. Area is additive while the cancellation of internal edge contributions follows from $Z_{K,\uparrow}+Z_{K,\downarrow}=Z^{(1)}(K,|)$ so that the specific formulae for $\G$ and $\G^2$ contributions follow for general rational simple polygons.
\end{proof}

\noindent{\bf Remark.} Although the conditions $f_1,f_3\not=0$ are used in the proof, they can be removed by transforming under an element of $SL_2(\Z)$.

\begin{lem} \textbf{\textit{(Gluing Lemma)}} Suppose that a rational simple polygon $\cP$ is split into two rational simple polygons $\cP_1, \cP_2$ by a common part of the boundary consisting of segments $\cL_i$, $i=1,\ldots,r$ with all of the boundary segments of $\cP_1$, $\cP_2$ not parallel to null directions of $f_h$. Then the deficit from additivity of $Z^{(0)}_K$ is given by
\[Z^{(0)}_K(\cP) = Z^{(0)}_K(\cP_1) + Z^{(0)}_K(\cP_2)- \sum_{i=1}^rZ_K^{(0)}(\cL_i)+q^{f(0,0)}\sum_{j=1}^d\delta_{K\bfv_j\in\Z^2}\]
where $\bfv_1,\ldots,\bfv_d$ are the interior vertices of $\bigcup_{i=1}^r\cL_i$, that is those end points of segments $\cL_i$ which appear as endpoints of two segments.
\end{lem}
\begin{proof}  Taking into account double counting,
$S_K(cP)=S_K(\cP_1)+S_K(\cP_2)-S_K\left(\bigcup_{i=1}^r\cL_i\right)$.
Meanwhile,
\[S_K\left(\bigcup_{i=1}^r\cL_i\right)=\sum_{i=1}^r{}S_K(\cL_i)-\sum_{j=1}^dq^{f(K\bfv_j)}\delta_{K\bfv_j\in\Z^2}\]
Since $f(K\bfv_j)\eK{}f(0,0)$ hence by Lemma \ref{modKmodphiK}, $q^{f(K\bfv_j)}\ePK{}q^{f(0,0)}$. Thus
\begin{align}
    S_K(\cP)\ePK{}S_K(\cP_1)+S_K(\cP_2)-\sum_{i=1}^r{}S_K(\cL_i)+q^{f(0,0)}\sum_{j=1}^d\delta_{K\bfv_j\in\Z^2}\label{gluingS}
\end{align}
Note that for all but a finite number of primes $K$, $\delta_{K\bfv_j\in\Z^2}=\delta_{\bfv_j\in\Z^2}$.
By Theorem \ref{2DquantumPick} applied to $\cP$, $\cP_1$ and $\cP_2$ and Theorem \ref{lineinvariant} applied to $\cL$ then for almost all $K$,
\begin{align*}
  S_K(\cP)&\ePK\G^2{}A(\cP)q^{\frac{D}{\Delta}}\left(\frac{(\Delta)_K}{K}\right)
  +\G\sum_{e\in\partial\cP} Z_{K,e}+Z^{(0)}_K(\cP)\\
  S_K(\cP_1)&\ePK\G^2{}A(\cP_1)q^{\frac{D}{\Delta}}\left(\frac{(\Delta)_K}{K}\right)
  +\G\sum_{e\in\partial\cP_1} Z_{K,e}+Z^{(0)}_K(\cP_1)\\
  S_K(\cP_2)&\ePK\G^2{}A(\cP_2)q^{\frac{D}{\Delta}}\left(\frac{(\Delta)_K}{K}\right)
  +\G\sum_{e\in\partial\cP_2} Z_{K,e}+Z^{(0)}_K(\cP_2)\\
  S_K(\cL_i)&\ePK\G{}Z^{(1)}_K(\cL_i)+Z^{(0)}_K(\cL_i)
\end{align*}
The edges in $\partial\cP$ are precisely those in $\partial\cP_1$ and $\partial\cP_2$, except for internal edges which do not appear in $\partial\cP$ but appear in the boundaries of $\cP_1,\cP_2$ with opposite orientations.  Comparing the $\G^2$, $\G^1$ and $\G^0$ terms in \eqref{gluingS} we see that the $\G$ and $\G^2$ terms precisely match because $A(\cP)=A(\cP_1)+A(\cP_2)$ while for any line segment $\cL$ which corresponds to oriented edges $e,-e$, then  $Z_{K,e}+Z_{K,-e}=Z^{(1)}(\cL)$. Hence the terms without $\G$ must also precisely match, which is the result of the lemma. Alternatively, one can skip the comparison of $\G$ and $\G^2$ terms and jump straight to the last line of the proof by uniqueness of $\G$ decompositions, Theorem \ref{uniquenessinarbitrarydimension}.
\end{proof}

\section{\textsc{Integral Representation of the Symbolic Integral}}
As a first step to part (5) of Conjecture \ref{quantumpickconjecture}, we give an integral formula for one-dimensional symbolic integrals.

\begin{lem} Let $f(z) \in \R [z]$ and $a,b \in \R$ with $|a-b|< 1 $ then,
\[
    \int_{a+ B }^{b+ B } f (z) dz =  \sin\pi (b-a)
    \int_{-\infty}^{\infty} \frac{ f\left(iz+\frac{a+b-1}{2}\right) }{ \cosh 2\pi{}z  + \cos \pi (b-a) } dz
\]
is an integral representation for the symbolic expression.
\end{lem}
\begin{proof} By linearity, it suffices to prove the equality for monomials $f(z)=z^m$, $m\in\N$. Set $\overline{B}_m (a,b)  :=  \int_{a+ B }^{b+ B  }z^m dz  = \frac{1}{m+1} \left[ B_{m+1} (b) - B_{m+1} (a) \right] $.  By Definition \ref{Bpolygenfn}, $\sum_{m=0 }^\infty  B_m (z) \frac{t^m }{m! }  = \frac{t e^{tz}}{ e^t -1 } $ for $|t|<2\pi$ and so shifting index $m=k+1$
\[
    \frac{t e^{tb} - t e^{ta}}{e^t -1 } = \sum_{ m=0 }^\infty \frac{  \left[ B_m (b) - B_m (a) \right]  t^m   }{m!}
    = t  \sum_{k = 0 }^\infty \frac{  \left[ B_{k+1} (b) - B_{k+1} (a)  \right]  t^k  }{(k+1)k!}  = t \sum_{k=0}^\infty  \overline{B}_k (a,b)  \frac{t^k }{k!}
\]
so that $ \sum_{m= 0 }^\infty\overline{B}_m (a,b) \frac{t^m }{m!} = \frac{e^{tb} - e^{ta}}{e^t -1 }=e^{ct} \frac{ \sinh { (dt) } }{\sinh{(t/2)}} $ for $c := \frac{a+ b -1}{2}$ and $ d := \frac{b-a}{2} $. For $|\alpha|<\pi$, the functions $\frac{\sinh{(\alpha{}x)}}{\sinh{(\pi x )}}$ and $\frac{1}{\sqrt{2\pi}} \frac{\sin\alpha}{\cosh{x} + \cos\alpha} $ form a Fourier pair (see \cite{N10} p.30). Hence
\begin{align*}
\frac{\sinh{(\alpha{}x)}}{\sinh{(\pi x )}} = \frac{1}{\sqrt{2 \pi }} \int_{-\infty}^\infty  \frac{1}{\sqrt{2 \pi }} \frac{\sin{\alpha}}{  \cosh {\omega} + \cos{\alpha} }  e^{ix\omega }d\omega
\end{align*}
Replacing $\alpha$, $x$, $\omega$ by $2\pi{}d$, $\frac{t}{2\pi}$, $2\pi{}z$ respectively, gives
\begin{align*}
    \frac{\sinh{d t }}{\sinh\frac{t}2} = \int_{-\infty}^\infty \frac{\sin{2\pi{}d}}{\cosh2\pi{}z + \cos{2\pi{}d} } e^{izt }dz.
\end{align*}
so that for $|t|<2\pi$,
\[
     \sum_{m= 0 }^\infty\overline{B}_m (a,b) \frac{t^m }{m!}=  e^{ct} \frac{ \sinh{dt} }{\sinh\frac{t}{2}} = \int_{-\infty}^{\infty} \frac{\sin2\pi{}d}{\cosh2\pi{}z + \cos2\pi{}d} e^{t(c+iz) } dz
\]
Comparing coefficients of $t^m$ gives
$\overline{B}_m(a,b)=\int_{-\infty}^{\infty} \frac{(c+iz)^m\sin2\pi{}d}{\cosh2\pi{}z + \cos2\pi{}d} e^{t(c+iz) } dz $ as required.
\end{proof}

\noindent A special case of this lemma for $a=\frac14$, $b=\frac34$ is the integral formula (2.1.5) in \cite{Law95} for Euler numbers
$E_m=2^{2m}\overline{B}_m(\frac14,\frac34) = 2^{2m+2}\int_{-\infty}^{\infty} \frac{ (iz)^m  }{e^{2\pi{}z}+e^{-2\pi{}z}}dz$.

This lemma allows symbolic integrals with integrands of the form $q^{g(z)}$ where $g(z)$ is a quadratic, to be considered as asymptotic expansions of integrals which are convergent for $|q|<1$ or for $|q|>1$, according to the sign of the leading term of $g(z)$. This is sufficient to deal with $Z^{(0)}_K$ in the one-dimensional case and the edge contributions $Z_{K,e}$ to $Z^{(1)}_K$ in the two-dimensional case. To deal with $Z^{(0)}_K(\cP)$ in the two-dimensional case requires a formula giving a holomorphic extension of repeated symbolic integrals.

\section{\textsc{Example}}
According to \cite{Law95}, a normalised version of the WRT invariant of $M_{n,t}$, the $\Q$HS given by surgery around a $(2,n)$-torus with $t$ additional twists ($t+n\not=0$) at a $K$-th root of unity $q$ is given by
\[
Z_K(M_{n,t})=(-1)^{t+\alpha}A^{2n}\frac{\sum_{(a,b)\in{}X}(-1)^bq^{g(a,b)}}{(q^{-\delta}-1)\sum_{a=0}^{K-1}q^{\delta{}a^2}}
\]
where $A=q^\frac14$, $\delta=\sgn(t+n)$, $X$ is the set of integral points in $[-\alpha,\alpha]^2$ with opposite parity, $\alpha=\frac{K-1}{2}$ and $g(a,b)=\frac{n}{8}\big((a+b)^2-1\big)+\frac{t-n}{16}(a-b)^2+\frac14(3a+b)$. The sum in the denominator is $\G$ for $\delta=1$ and $\left(\frac{-1}{K}\right)\G=(-1)^\alpha\G$ for $\delta=-1$. Meanwhile, for the numerator, split $X$ into points $(2i+1,2j), (2j,2i+1)$ where for $K\equiv_41$ the range is $\frac{1-K}{4}\leq{}i\leq\frac{K-5}{4}$, $\frac{1-K}4\leq{}j\leq\frac{K-1}{4}$ while for $K\equiv_43$ the range is $\frac{-1-K}{4}\leq{}i\leq\frac{K-3}{4}$, $\frac{3-K}4\leq{}j\leq\frac{K-3}4$; the summand now becomes $q^{g(2i+1,2j)}-q^{g(2j,2i+1)}$. The range differs from $(i,j)\in{}K[-\frac14,\frac14]^2$ by requiring those points with $i=\frac{K-1}{4}$ to be removed in the case of $K\equiv_41$ and those points with $i=\frac{-1-K}4$ to be added in the case of $K\equiv_43$. In either case, the summand for the exceptional points is $q^{g(\frac{\pm{}K+1}2,2j)}-q^{g(2j,\frac{\pm{}K+1}2)}$ which is congruent modulo $\Phi_K$ to $q^{g(\frac12,2j)}-q^{g(2j,\frac12)}$. For simplicity, we complete the computation just for the case of the Poincar\'e homology sphere, $n=-3$, $t=2$, $\delta=-1$.
\begin{align}
q(q-1)Z_K\ePK{}A^{-2}\G^{-1}\left[\sum_{(i,j)\in{}\Z^2\cap{}K\big[-\frac14,\frac14\big]^2}(q^{g(2i+1,2j)}-q^{g(2j,2i+1)})
+(-1)^\alpha\!\!\!\!\sum_{j\in{}\Z\cap{}K\big[-\frac14,\frac14\big]}(q^{g(2j,\frac12)}-q^{g(\frac12,2j)})\right]\label{PoincareHSsum}
\end{align}
For the double sum, apply Theorem \ref{2DquantumPick} to $\cP=\big[-\frac14,\frac14\big]^2$ with the two quadratic polynomials
\begin{align*}
    f_1(i,j)&\equiv{}g(2i+1,2j)=-\frac14i^2-\frac{11}2ij-\frac14j^2+\frac54i-\frac94j+\frac{17}{16}\\
    f_2(i,j)&\equiv{}g(2j,2i+1)=-\frac14i^2-\frac{11}2ij-\frac14j^2+\frac14i-\frac54j+\frac{9}{16}
\end{align*}
in turn, in place of $f$. Both have $\Delta=-30$, $D=-\frac{59}4$ while for the first quadratic  $z_0=-\frac{13}{30}$, $w_0=\frac4{15}$ and for the second quadratic $z_0=-\frac7{30}$, $w_0=\frac1{15}$. This gives level two $\G$ decompositions (see \eqref{Zright}, \eqref{Zleft}, \eqref{Zup}, \eqref{Zdown}, \eqref{Z0square})
\begin{align*}
&\sum_{(i,j)\in{}\Z^2\cap{}K[-\frac14,\frac14]^2}q^{g(2i+1,2j)}
\ePK\frac14q^{\frac{59}{120}}\left(\frac{-30}K\right)\G^2
+\int_{\{\frac{K}4\}+B}^{1-\{\frac{K}4\}+B}\int_{\{\frac{K}4\}+B}^{1-\{\frac{K}4\}+B}q^{f_1(z,w)}dwdz\\
&+\frac12q^{\frac{59}{120}}(-1)^\alpha\G
\left[\int_B^{\frac{43}{30}-\{\frac{K}4\}+B}q^{30z^2}dz
      +\!\!\!\int_B^{\frac{11}{15}-\{\frac{K}4\}+B}q^{30w^2}dw
      -\!\!\!\int_B^{\frac{13}{30}+\{\frac{K}4\}+B}q^{30z^2}dz
      -\!\!\!\int_B^{-\frac{4}{15}+\{\frac{K}4\}+B}q^{30w^2}dw\right]\\
&\sum_{(i,j)\in{}\Z^2\cap{}K[-\frac14,\frac14]^2}q^{g(2j,2i+1)}
\ePK\frac14q^{\frac{59}{120}}\left(\frac{-30}K\right)\G^2
+\int_{\{\frac{K}4\}+B}^{1-\{\frac{K}4\}+B}\int_{\{\frac{K}4\}+B}^{1-\{\frac{K}4\}+B}q^{f_2(z,w)}dwdz\\
&+\frac12q^{\frac{59}{120}}(-1)^\alpha\G
\left[\int_B^{\frac{37}{30}-\{\frac{K}4\}+B}q^{30z^2}dz
      +\!\!\!\int_B^{\frac{14}{15}-\{\frac{K}4\}+B}q^{30w^2}dw
      -\!\!\!\int_B^{\frac{7}{30}+\{\frac{K}4\}+B}q^{30z^2}dz
      -\!\!\!\int_B^{-\frac{1}{15}+\{\frac{K}4\}+B}q^{30w^2}dw\right]
\end{align*}
Note that when $K\equiv_41$, then $\{\frac{K}4\}=\frac14$ and $1-\{\frac{K}4\}=\frac34$ while when $K\equiv_43$, then $\{\frac{K}4\}=\frac34$ and $1-\{\frac{K}4\}=\frac14$ so that these $\G$ decompositions simplify to
\begin{align*}
    \frac14q^{\frac{59}{120}}\left(\frac{-30}K\right)\G^2
+\frac12q^{\frac{59}{120}}\G
\left[\int_{\frac{41}{60}+B}^{\frac{71}{60}+B}q^{30z^2}dz
 +\!\!\!\int_{-\frac1{60}+B}^{\frac{29}{60}+B}q^{30w^2}dw\right]
+\int_{\frac14+B}^{\frac34+B}\int_{\frac14+B}^{\frac34+B}q^{f_1(z,w)}dwdz\\
       \frac14q^{\frac{59}{120}}\left(\frac{-30}K\right)\G^2
+\frac12q^{\frac{59}{120}}\G
\left[\int_{\frac{29}{60}+B}^{\frac{59}{60}+B}q^{30z^2}dz
      +\!\!\!\int_{\frac{11}{60}+B}^{\frac{41}{60}+B}q^{30w^2}dw\right]
      +\int_{\frac14+B}^{\frac34+B}\int_{\frac14+B}^{\frac34+B}q^{f_2(z,w)}dwdz
\end{align*}
and in the difference, the $\G^2$ term will cancel.
The one-dimensional sum in \eqref{PoincareHSsum} involves exponents $h_1(j)\equiv{}g(2j,\frac12)=-\frac14j^2+\frac18j+\frac{31}{64}$ and $h_2(j)\equiv{}g(\frac12,2j)=-\frac14j^2-\frac78j+\frac{47}{64}$ which have critical values $\frac12$ and $\frac32$ respectively. By Theorem \ref{1dqpt}, we have level one $\G$ decompositions
\begin{align*}
    \sum_{j\in{}\Z\cap{}K\big[-\frac14,\frac14\big]}q^{g(2j,\frac12)}
    &\ePK\frac12q^{\frac12}(-1)^\alpha\G+\int_{\{\frac{K}4\}+B}^{1-\{\frac{K}4\}+B}q^{h_1(z)}dz
    =\frac12q^{\frac12}(-1)^\alpha\G+(-1)^\alpha\int_{\frac14+B}^{\frac34+B}q^{h_1(z)}dz\\
    \sum_{j\in{}\Z\cap{}K\big[-\frac14,\frac14\big]}q^{g(\frac12,2j)}
    &\ePK\frac12q^{\frac32}(-1)^\alpha\G+\int_{\{\frac{K}4\}+B}^{1-\{\frac{K}4\}+B}q^{h_2(z)}dz
    =\frac12q^{\frac32}(-1)^\alpha\G+(-1)^\alpha\int_{\frac14+B}^{\frac34+B}q^{h_2(z)}dz
\end{align*}
The $\G^0$ terms in the combination of the $\G$ decompositions of the four sums in  \eqref{PoincareHSsum} is
\begin{align}
\int_{\frac14+B}^{\frac34+B}\int_{\frac14+B}^{\frac34+B}(q^{f_1(z,w)}-q^{f_2(z,w)})dwdz\label{G0terms}
+\int_{\frac14+B}^{\frac34+B}(q^{h_1(z)}-q^{h_2(z)})dz
\end{align}
which we claim vanishes. To see this, note that since $f_1(z,w)=g(2z+1,2w)$, we can change variables $z:=z+\frac12$
\[
    \int_{\frac14+B}^{\frac34+B}\int_{\frac14+B}^{\frac34+B}q^{f_1(z,w)}dwdz
       =\int_{\frac34+B}^{\frac54+B}\int_{\frac14+B}^{\frac34+B}q^{g(2z,2w)}dwdz
       =\int_{\frac14+B}^{\frac34+B}q^{h_2(w)}dw-\int_{\frac14+B}^{\frac34+B}\int_{\frac14+B}^{\frac34+B}q^{g(2z,2w)}dwdz
\]
where in the second step we used that $\int_{a+\frac12+B}^{a+1+B}h(z)dz=h(a)-\int_{a+B}^{a+\frac12+B}h(z)dz$, a special case of the Euler-Maclaurin formula, with $a=\frac14$ and $h(z)$ being the inner integral.
Similarly,
\[
\int_{\frac14+B}^{\frac34+B}\int_{\frac14+B}^{\frac34+B}q^{f_2(z,w)}dwdz
=\int_{\frac14+B}^{\frac34+B}q^{h_1(w)}dw-\int_{\frac14+B}^{\frac34+B}\int_{\frac14+B}^{\frac34+B}q^{g(2w,2z)}dwdz
\]
Note that the double integrals in the last two equations are identical, by changing the order of symbolic integration on the square; from this it follows that the $\G^0$ contribution \eqref{G0terms} vanishes. This leaves only the $\G^1$ terms in the combination of the $\G$ decompositions of the four sums in  \eqref{PoincareHSsum}, so that
\[
q(q-1)Z_K\ePK\frac12q^{-\frac1{120}}\left[\int_{\frac{41}{60}+B}^{\frac{71}{60}+B}q^{30z^2}dz
 +\!\!\!\int_{-\frac1{60}+B}^{\frac{29}{60}+B}q^{30w^2}dw
 -\int_{\frac{29}{60}+B}^{\frac{59}{60}+B}q^{30z^2}dz
      -\!\!\!\int_{\frac{11}{60}+B}^{\frac{41}{60}+B}q^{30w^2}dw
 \right]+\frac12(1-q)
\]
Applying $\int_{a+\frac12+B}^{a+1+B}h(x)dx=h(a)-\int_{a+B}^{a+\frac12+B}h(x)dx$ to transform the first and third integrals in the last formula to match the second and fourth integrals will acquire extra terms $q^\frac{121}{120}-q^\frac1{120}$, which exactly cancel the term $\frac12(1-q)$ and gives the simplified expression
\[
q(q-1)Z_K\ePK{}q^{-\frac1{120}}\left[\int_{-\frac1{60}+B}^{\frac{29}{60}+B}q^{30w^2}dw
 -\int_{\frac{11}{60}+B}^{\frac{41}{60}+B}q^{30w^2}dw
 \right]
\]
This should be compared with \cite{LZ} where it was shown that the rescaled quantum invariant $W_K=q(q-1)Z_K$ of the Poincar\'e homology sphere at a $K$-th root of unity is the radial limit at that root of unity of the holomorphic function
\begin{align}
    W_\infty(q)\equiv1-\frac12\sum_{n=1}^\infty\chi_+(n)q^{\frac{n^2-1}{120}}\label{Winfty}
\end{align}
defined on $|q|<1$ where $\chi_+:\N\to\{-1,0,1\}$ is defined by $\chi_+(n)=+1$ for $n\equiv_{60}1,11,19,29$, $\chi_+(n)=-1$ for $n\equiv_{60}31,41,49,59$ and $\chi_+(n)=0$ otherwise. Further, in \cite{LZ}, the asymptotic expansion of $W_\infty(q)$ around $q=1$ was identified with the similarly rescaled Ohtsuki invariant, $q(q-1)Z_\infty(M)$ and the infinite sum on the right hand side of \eqref{Winfty} was identified as the Eichler integral of a suitable weight $\frac32$ theta function.

\end{document}